\documentclass[12pt,reqno]{amsart}
\usepackage{fullpage}
\usepackage{graphicx}

\def\ch{\mathop{\hbox{\rm ch}}\nolimits}

\def\g{\mathfrak g}

\def\fraka{\mathfrak a}

\def\k{\mathfrak k}
\def\p{\mathfrak p}

\def\R{\mathbb{R}}
\def\C{\mathbb{C}}

\def\a{\mathfrak a}

\def\ss1{\mathfrak s_{\overline 1}}

\def\hs1{\mathfrak h_{\overline 1}}

\def\Pg{\mathrm{P}}

\newcommand{\eps}{\varepsilon}

\def\G{\mathrm{G}}
\def\N{\mathrm{N}}

\def\Qg{\mathrm{Q}}

\def\K{\mathrm{K}}

\def\M{\mathrm{M}}

\def\Sg{\mathrm{S}}

\def\L{\mathrm{L}}
\def\Bbb{\mathbb}

\def\N{\mathrm{N}}
\def\A{\mathrm{A}}

\def\X{\mathrm{X}}
\def\GL{\mathrm{GL}}
\def\SL{\mathrm{SL}}
\def\SO{\mathrm{SO}}
\def\SU{\mathrm{SU}}

\def\Ug{\mathrm{U}}

\def\Dg{\mathrm{D}}
\def\Eg{\mathrm{E}}

\def\B{\mathrm{B}}
\def\Wg{\mathrm{W}}

\def \wt{\widetilde}

\def\W{\mathsf{W}}

\def\X{\mathsf{X}}
\def\Y{\mathsf{Y}}

\newcommand{\cHC}{c_{\scriptscriptstyle \mathrm{HC}}}

\def\diag{\mathop{\hbox{\rm diag}}\nolimits}

\def\Re{\mathop{\rm Re}\nolimits}
\def\Im{\mathop{\hbox{\rm Im}}\nolimits}

\def\sign{\mathop{\hbox{\rm sign}}\nolimits}

\newcommand\floor[1]{\lfloor#1\rfloor}

\def\lim{\mathop{\hbox{\rm lim}}\nolimits}

\def\Res{\mathop{\hbox{\rm Res}}\limits}
\def\Resn{\mathop{\hbox{\rm Res}}\nolimits_n \wt R}

\def\Ind{\mathop{\hbox{\rm Ind}}\nolimits}

\def\th{\mathop{\hbox{\rm th}}\nolimits}
\def\ch{\mathop{\hbox{\rm ch}}\nolimits}
\def\sh{\mathop{\hbox{\rm sh}}\nolimits}

\newcommand{\cz}{{\mathop{\mathrm c}}}
\newcommand{\sz}{{\mathop{\mathrm s}}}

\renewcommand{\l}{\lambda}
\newcommand{\lrw}{\lambda(r,w)}
\newcommand{\lzw}{\lambda(z,w)}
\newcommand{\la}{\lambda_\alpha}
\newcommand{\compl}{\mathbf{i}}

\newcommand{\inner}[2]{\langle#1,#2\rangle}

\newcommand{\rhoXtwo}{\rho_\X^2}
\newcommand{\rhoX}{\rho_\X}
\newcommand{\bxi}{\boldsymbol{\xi}}

\newcommand{\Rlog}{R_{\rm log}}
\newcommand{\Ulog}{U_{\rm log}}

\def\fonttitre{\textsf}
\newcounter{thh}

\newtheorem{thm}[thh]{\fonttitre{Theorem}}

\newtheorem{pro}[thh]{\fonttitre{Proposition}}
\newtheorem*{pro*}{\fonttitre{Proposition}}
\newtheorem{cor}[thh]{\fonttitre{Corollary}}
\newtheorem*{coro*}{\fonttitre{Corollary}}
\newtheorem{lem}[thh]{\fonttitre{Lemma}}
\theoremstyle{definition}
\newtheorem{rem}{\fonttitre{Remark}}

\newtheorem*{nota*}{\fonttitre{Notation}}
\newenvironment{prf}{\begin{proof}}{\end{proof}}
\def\muet{ \ifthenelse{\equal{a}{b}}}
\def\nn{\nonumber}
\newcommand{\thmlist}{
\renewcommand{\theenumi}{\alph{enumi}}
\renewcommand{\labelenumi}{(\theenumi)}}

\begin{document}
\thanks{The first and second author would like to thank the University of Oklahoma for hospitality and financial support.
The third author gratefully acknowledges hospitality and financial support from the Universit\'e de Lorraine and partial support from the NSA grant H98230-13-1-0205. }

\makeatletter
\title[Resonances for the Laplacian on $\SL(3,\R)/\SO(3)$]
{Resonances for  the Laplacian \\on Riemannian symmetric spaces:\\
the case of $\SL(3,\R)/\SO(3)$ }
\author{J. Hilgert}
\address{Department of Mathematics,
Paderborn University,
Warburger Str. 100,
D-33098 Paderborn,
Germany}
\email{hilgert@math.uni-paderborn.de}
\author{A. Pasquale}
\address{Universit\'e de Lorraine, Institut Elie Cartan de Lorraine, UMR CNRS 7502, Metz, F-57045, France}
\email{angela.pasquale@univ-lorraine.fr}
\author{T. Przebinda}
\address{Department of Mathematics, University of Oklahoma, Norman, OK 73019, USA}
\email{tprzebinda@ou.edu}


\date{}
\subjclass[2010]{Primary: 43A85; secondary: 58J50, 22E30}
\keywords{}

\begin{abstract}
We show that the resolvent of the Laplacian on $\SL(3,\R)/\SO(3)$ can be lifted to a meromorphic function on a Riemann surface which is a branched covering of $\mathbb C$. The poles of this function are called the resonances of the Laplacian. We determine all resonances and show that the corresponding residue operators are given by convolution with spherical functions parameterized by the resonances. The ranges of these operators are infinite dimensional irreducible $\SL(3,\R)$-representations. We determine their Langlands parameters and wave front sets. Also, we show that precisely one of these representations is unitarizable. Alternatively, they are given by the differential equations which determine the image of the Poisson transform associated with the resonance. 
\end{abstract}

\maketitle

\tableofcontents

\section*{Introduction}
\label{section:introduction}

The notion of resonance was introduced in quantum mechanics to study metastable states
of a system, that is long-lived states from which the system deviates only with sufficiently
strong disturbances. Mathematically, resonances replace discrete eigenvalues of linear operators on non-compact  domains and appear as poles of the meromorphic continuation of their resolvents.

The mathematical study of resonances initiated for Schr\"odinger operators
on $\R^n$. Later, it was extended to more geometric situations, such as the Laplacian on hyperbolic and asymptotically hyperbolic manifolds, symmetric or locally symmetric spaces, and Damek-Ricci spaces. In a typical situation, one works on a complete Riemannian manifold $\X$, for which the positive Laplacian $\Delta$ is an essentially self-adjoint operator on the Hilbert space $L^2(\X)$ of square integrable functions on $\X$. Suppose that  $\Delta$ has a continuous spectrum $[\rhoXtwo,+\infty)$, with $\rhoXtwo \geq 0$.
The resolvent $R(z)=(\Delta_\X-\rhoXtwo-z^2)^{-1}$ of the shifted Laplacian (or Helmholtz operator) $\Delta-\rhoXtwo$ is then a holomorphic function of $z$ on the upper (and on the lower) complex halfplane,
with values in the space of bounded linear operators on $L^2(\X)$. Let the resolvent act, not on the entire $L^2(\X)$, but on a dense subspace of $L^2(\X)$, for instance the space $C_c^\infty(\X)$ of compactly supported smooth functions on $\X$ or on some suitable weighted $L^2$ space. Then the map $z \mapsto R(z)$ might admit a meromorphic extension across $\R$ to a larger domain in $\C$ or to a cover of such a domain. The poles, if they exist, are the resonances of $\Delta-\rhoXtwo$. Sometimes also the name scattering poles is used, but the two concepts are not completely synonymous (see e.g. \cite{Gu05}). The basic questions concern the existence of the meromorphic extension of the resolvent, the distribution and counting properties of the resonances, the rank and interpretation of the residue operators associated with the resonances.

Let $\G$ be a connected noncompact real semisimple Lie group with finite center and let $\K$ be a maximal compact subgroup of $\G$. Then the homogeneous space $\X=\G/\K$ is a Riemannian symmetric space of the noncompact type. It is a complete Riemannian manifold with respect to its canonical $\G$-invariant Riemannian structure. The positive Laplacian is the opposite of the Laplace-Beltrami operator.  An important example of such spaces are the real hyperbolic spaces
$\mathbb H^n=\SO_0(n,1)/\SO(n)$. The resonances of the positive Laplacian on $\mathbb H^n$ have been studied by
Guillop\'e and Zworski,  \cite{GZ95}; see also \cite{Z06}. They proved that there are no resonances for $n$ even; for $n$ odd, there are resonances (which are explicitely determined) and the corresponding residue operators have finite rank.

The study of the analytic extension of the resolvent of the Laplacian for general Riemannian symmetric spaces of the noncompact type $\X=\G/\K$ was started by Mazzeo and Vasy, \cite{MV05}. The motivations were of different nature. First of all, these spaces form a natural class of complete Riemannian manifolds for which the geometric properties are well understood. Moreover, the analytic properties of their Laplace-Beltrami operator play an important role in representation theory and number theory. Furthermore, the radial component of the Laplace-Beltrami operator on a maximal flat subspace is a many-body type Hamiltonian, with the walls of Weyl chambers of the maximal flat corresponding to the collision planes. This suggested that many-body methods of geometric scattering theory could have been appropriate to this setting.
More precisely, the analysis carried out in \cite{MV05} combines microlocal techniques and an adaptation of the scattering method of complex scaling of Aguilar-Balslev-Combes, see e.g. \cite{HS96}. A different point of view, using the Helgason-Fourier analysis, was employed by A. Strohmaier, \cite{Str05}, and by Hilgert and Pasquale, \cite{HP09}. A further approach, using asymptotics of solutions the Laplacian on Damek-Ricci spaces, was employed by Miatello and Will in \cite{MW00}.

For a general Riemannian symmetric space of the noncompact type $\X=\G/\K$, Mazzeo and Vasy \cite{MV05} and Strohmaier \cite{Str05} independently proved that the resolvent
$z \mapsto R(z^2)$ of the shifted Laplacian $\Delta_\X-\rhoXtwo$ admits a holomorphic extension across $\R$. The domain of the extension depends on the parity of the real rank of the symmetric space $\X$ (i.e. the dimension of the maximal flat subspace of $\X$). This dependence on the parity of the dimension parallels the case of the Laplacian on $\R^n$ (see e.g. \cite[\S 1,6]{Mel95}).

Despite the many articles studying resonances on complete Riemannian manifolds, detailed information on the existence and nature of the resonances for the Laplacian on $\X$ is so far available only in the so-called even multiplicity and real rank one cases. The even multiplicity case corresponds to the situation in which the Lie algebra of $G$ has a unique class of Cartan subalgebras. This happens for instance when $G$ possesses a complex structure. In the even multiplicity case, the resolvent has an entire extension to a suitable covering of the complex plane, see \cite[Theorem 3.3]{Str05}.  So there are no resonances in this case. The general rank-one case was considered, with different approaches, in \cite{MW00} and \cite{HP09}: unless the symmetric space has even multiplicities (in which case there are no resonances), the
map $z \mapsto R(z^2)$ admits a meromorphic extension to $\C$ with simple poles along the negative imaginary axis. The poles are at the points $\zeta_k=|\alpha|\lambda_k$, $k\in \mathbb N$, where $|\alpha|$ is a constant depending on the normalization of the Riemannian measure and the $\lambda_k$'s range among the spectral parameters of the spherical functions $\varphi_{\lambda_k}$ on $\X$ which are matrix coefficients of finite dimensional spherical representations of $\G$. The resolvent residue operator at $\zeta_k$ is a constant multiple of the convolution operator by $\varphi_{\lambda_k}$ and its image is the space of the corresponding finite dimensional spherical representation. In particular, the rank of the residue operators is finite. See \cite[Theorem 3.8]{HP09}.

In \cite{MV04} and \cite{MV07}, Mazzeo and Vasy considered the specific case of $\X=\SL(3,\R)/\SO(3)$, to exemplify their microlocal and complex-scaling methods of analytic extension of the resolvent of the positive Laplacian. The space $\SL(3,\R)/\SO(3)$ is a symmetric space of real rank-two which can be realized as the space of symmetric 3-by-3 positive definite matrices with determinant $1$. Restricted to a maximal flat, its Laplace-Beltrami operator is a Calogero-Moser-Sutherland 3-body Hamiltonian of type II associated with the root system $\A_2$, see e.g. \cite[(3.1.14) and (3.8.3)]{Pe90}. The analysis of \cite{MV04} and \cite{MV07} left nevertheless open the basic questions on the existence and nature of resonances and resolvent residue operators.

In this paper we provide complete answers to these questions. In a first step one notices that for fixed $f\in C_c^\infty(\X)$ and $y\in\X$ the resolvent function $z\mapsto (R(z)f)(y)$ extends holomorphically to
$\C \setminus \big((-\infty,0] \cup i(-\infty, -\frac{1}{2}\rhoX]\big)$. The cut $(-\infty,0]$ leads to a logarithmic Riemann surface covering $\C \setminus \big(i(-\infty, -\frac{1}{2}\rhoX]\cup \{0\}\cup  i[\tfrac{1}{2}\rhoX,+\infty)\big)$ to which $z\mapsto(R(z)f)(y)$ can be lifted holomorphically (see Corollary~\ref{cor:logRiemann}). This narrows down the location of the resonances to the negative imaginary axis. Our main result, Theorem~\ref{thm:meroextshiftedLaplacian}, then says that for each $N\in \mathbb N$ there is an open neighborhood of $i(0,N+1)$ together with a branched cover $M_{(\gamma_N)}$ to which $z\mapsto (R(z)f)(y)$ can be lifted meromorphically. Poles can occur only above the points $-i(\mathbb N+\frac{1}{2})$, and they are of order at most one. For special $f$ and $y$ the lift may be holomorphic at some of these points. Obviously this is the case for $f=0$. The residues of the lifted functions may be calculated and they are given as convolution of $f$ with the spherical function whose spectral parameter is the resonance. Unlike the rank one case, the residue resolvent operators are not of finite rank (see Proposition~\ref{prop:eigenspaces}). On the other hand, as in the rank one case the range of each of the residue operator is a $\G$-representation which can be identified explicitly. More precisely, it is the unique irreducible subquotient of a (non-unitary) principal series whose Langlands parameter can be read off from the resonance (see Proposition~\ref{The range of the residue operator}). 
Since the unitary dual of $\SL(3,\R)$ is known, we are able to detect the unique resonance for which the corresponding representation is unitarizable. Unfortunately the spherical unitary dual of some other real rank two semisimple groups was not classified yet. Thus in these cases the question of unitarizibility of the residue representations will be more difficult. We are not aware of any place in the literature where the authors actually prove that any residue operators have infinite rank. Typically one uses  analytic Fredholm theory to show that these operators are of finite rank. The Fredholm theory is not applicable in the case we consider. Instead we use Langlands classification to show that the rank is infinite.

Alternatively, the range of the residue operator at some resonance is given by the differential equations which determine the image of the Poisson transform associated with this resonance (see Remark~\ref{rem:Poisson2}). 

In order to compute the residues, initially we tried to reduce the problem from rank two to the rank one, considered in \cite{HP09}, by pursuing a double rank-one integration and deforming the real line to an unbounded cycle in the complex plane. However this method leads to technical difficulties. Instead we decided to use polar coordinates, deforming a circle to an ellipse, see section \ref{subsection:deformationS1}, and get directly to the result. Since the Laplacian has rotational symmetry this is a natural approach, which was used in the Euclidean case to show that there are no residues on $\X=\R^2$ despite the fact that there is a simple pole at zero for $\X=\R$, see \cite{Mel95}.
These computations lead quickly to a local meromorphic extensions of the resolvent, see section \ref{subsection:meroextGn}. We glue them together in order to get an explicit global extension. This is a non-trivial process, described in sections \ref{subsection:mero-ext-F} and \ref{subsection:mero-ext-resolvent}.

\subsection*{Notation}
We shall use the standard notation $\mathbb Z$, $\mathbb N$, $\R$,  $\R^+$, $\C$ and $\C^\times$ for the integers, the nonnegative integers, the reals, the positive reals, the complex
numbers and the non-zero complex numbers, respectively. We also set $\R^-=-\R^+$. If $\X$ is a manifold, then $C^\infty(\X)$ and $C^\infty_c(\X)$ respectively denote the space of smooth functions and the space of smooth compactly supported functions on $\X$.

\section{Preliminaries}\label{section:preliminaries}
\subsection{Structure of $\X=\SL(3,\R)/\SO(3)$}
\label{subsection:structureSL3}
Let $\G=\SL(3,\R)$ be the Lie group of 3-by-3 real matrices of determinant 1. The Lie algebra $\g=\mathfrak{sl}(3,\R)$ of $\G$ consists of the 3-by-3 matrices with real coefficients and trace equal to $0$. The $\pm 1$-eigenspace decomposition of $\g$ with respect to the Cartan involution $\theta(x)=-x^t$, where $\null\cdot^t$ denote transposition, yields the Cartan decomposition $\g=\k\oplus \p$. Here $\k=\mathfrak{so}(3)$ is the Lie algebra of skew-symmetric 3-by-3 matrices and  $\p$ is the vector subspace of symmetric matrices in $\g$.

We consider $\X=\G/\K$ as a symmetric space endowed with the $\G$-invariant Riemannian
metric associated with the Cartan-Killing form of $\mathfrak{g}$. It can be realized as the space $\widetilde \X$ of the 3-by-3 symmetric positive definite matrices with determinant 1. Indeed,
$\G$ acts transitively on $\widetilde \X$, the element $g\in \G$ acting as the isometry $x\mapsto gxg^t$ of $\widetilde \X$, and $\K$ is the isotropy subgroup of the identity matrix.

Choose
$$\fraka=\{\diag(a_1,a_2,a_3); a_j \in \R, \sum_{j=1}^3 a_j=0\}$$
as a maximal abelian subalgebra in $\p$. For $j=1,2,3$ define $\varepsilon_j \in \fraka^*$ by $$\varepsilon_j(\diag(a_1,a_2,a_3))=a_j\,.$$
Then the set $\Sigma$ of (restricted) roots of $(\g,\fraka)$ is
\begin{equation}
\label{eq:SigmaSL3}
\Sigma=\{\alpha_{i,j}=\varepsilon_i-\varepsilon_j: 1\leq i \neq j \leq 3\}\,.
\end{equation}
It is a root system of type $\A_2$. All root multiplicities $m_\alpha$ are equal to $1$.
Take $\Sigma^+=\{\alpha_{i,j}: 1\leq i < j \leq 3\}$ as a set of positive roots. The corresponding system of simple roots is $\Pi=\{\alpha_{1,2}, \alpha_{2,3}\}$. The element $\rho=\frac{1}{2} \sum_{\alpha \in \Sigma^+} m_\alpha \alpha \in \mathfrak{a}^*$ is equal to
$$
\rho=\alpha_{1,2}+\alpha_{2,3}=\alpha_{1,3}\,.
$$
We denote by $\fraka^+$ the positive Weyl chamber associated with $\Sigma^+$ and by
$\fraka^*_+$ the elements in the real dual space $\fraka^*$ of $\fraka$ which are positive on
$\fraka^+$.
The Weyl group $\Wg$ of $\Sigma$ is $S_3$ acting
as the group of permutations of the three elements $\varepsilon_1, \varepsilon_2, \varepsilon_3$.

In the following, we denote by the same symbol $\inner{\cdot}{\cdot}$  the restriction of the Cartan-Killing form of $\g$ to $\fraka\times \fraka$, the dual inner product on $\fraka^*$ and their $\C$-bilinear extensions to the complexifications $\fraka_\C$ and $\fraka^*_\C$, respectively.
Explicitly, $\inner{\cdot}{\cdot}$ is given by $\inner{\varepsilon_i}{\varepsilon_j}=
6\delta_{i,j}$ for all $i,j=1,2,3$. Hence, $\inner{\alpha}{\alpha}=12$
for every $\alpha \in \Sigma$. Moreover, $\inner{\alpha_{1,2}}{\alpha_{2,3}}=-6$.

In later sections of this paper we will find it convenient to identify $\fraka^*$ with $\C$ by choosing a suitable basis.
To distinguish the resulting complex structure in $\fraka^*$ from the natural complex structure of $\fraka^*_\C$,
we shall indicate the complex units in $\fraka^*\equiv \C$ and $\fraka_\C^*$ by $i$ and $\compl$, respectively.
So $\fraka^*\equiv \C=\R +i \R$, whereas $\fraka_\C^*=\fraka^*+\compl\fraka^*$.
For $r,s\in \R$ and $\lambda,\nu\in \fraka^*$ we have $(r+is)(\lambda+\compl\nu)=(r\lambda-s\nu)+\compl(r\nu+s\lambda)\in \fraka^*_\C$.

\subsection{Spherical representations}
\label{subsection:spherical-reps}
In the following, we denote by $\varphi_\lambda$ Harish-Chandra's spherical function of spectral parameter $\lambda\in \mathfrak{a}_\C^*$; see e.g. \cite[Chapter 3]{GV} or \cite[Chapter IV] {He2}.
The group $\G_\C=\SL(3,\C)$ is the simply connected complexification of $\G$. Its Lie algebra contains  
$\fraka_\C$, which is formed by the $3$-by-$3$ complex diagonal matrices with $0$ trace.
We consider the finite dimensional holomorphic representations $\pi_\mu$ of $\G_\C$ with highest weight $\mu$ relative to $\fraka_\C$, with dominance defined by $\Sigma^+$. Let $\mu$ be the restriction to $\fraka$ of such a weight. Then $\mu\in \fraka^*$. Recall that
 $\pi_\mu$ is said to be spherical, if there exists a non-zero vector $v_\mu$  in the representation space of $\pi_\mu$ which is $\K$-fixed, i.e. so that $\pi_\mu(k)v_\mu=v_\mu$ for all $k \in \K$. According to the Cartan-Helgason's theorem, $\pi_\mu$ is spherical if and only if $\mu_\alpha\in \mathbb N$ for all $\alpha \in \Sigma^+$.
In this case, the vector $v_\mu$ is unique up to constant multiples.

Thus $\mu \in \mathfrak{a}^*_\C$ is the highest restricted weight of a finite-dimensional spherical representation of $G$ if and only
if there exists $n_1,n_2 \in \mathbb{Z}_+$ so that
$\mu =n_1w_{1,2}+n_2w_{2,3}$. Here $w_{1,2},w_{2,3} \in \mathfrak{a}^*$ are the fundamental restricted weights, which are defined
by the conditions
$$\frac{\inner{w_{i,j}}{\alpha_{k,l}}}{\inner{\alpha_{k,l}}{\alpha_{k,l}}}=\delta_{(i,k),(j,l)}$$
for $(i,j),(k,l) \in \{(1,2),(2,3)\}$. Hence
\begin{equation}
\label{eq:omega}
w_{1,2}=\tfrac{2}{3}(2\alpha_{1,2}+\alpha_{2,3})
\qquad \textrm{and} \qquad
w_{2,3}=\tfrac{2}{3}(\alpha_{1,2}+2\alpha_{2,3})\,.
\end{equation}
Observe that if $\lambda \in \fraka^*_\C$ is written as
\begin{equation}
\label{eq:lambda-omega}
\lambda=\lambda_{1,2} w_{1,2}+ \lambda_{2,3} w_{2,3}\,,
\end{equation}
then $\lambda_{i,j}=\lambda_{\alpha_{i,j}}$, where for $\alpha \in \Sigma$ the numbers
$\la$ are defined by (\ref{eq:la}). For instance, for  $\lambda=\rho$, we have
$\rho=\tfrac{1}{2}(w_{1,2}+w_{2,3})$ as
$\rho_{1,2}=\rho_{2,3}=\tfrac{1}{2}$.

Let $\Y=\Ug/\K$ be the simply connected Riemannian symmetric space of compact type which is dual to $\X=\G/\K$. Then $\Ug=\SU(3)$.
Let $\pi_\mu$ be the finite-dimensional spherical representation of $\G_\C$ of highest
restricted weight $\mu$, and let $v_\mu$ be a $\K$-fixed vector in the space of $\pi_\mu$
having norm one in the inner product $(\cdot,\cdot)$ making the restriction of $\pi_\mu$ to $\Ug$ unitary. Then the matrix coefficient $g\mapsto (\pi_\mu(g)v_\mu,v_\mu)$ is a $\K_\C$-bi-invariant holomorphic function of $g\in \G_\C$. Considered as a function on $\Y$, it is the spherical function on $\Y$ of spectral parameter $\mu$. Considered as a function on $\X$, it agrees with the spherical function $\varphi_{\mu+\rho}$.

\subsection{Eigenspace representations}
\label{subsection:eigenspace-reps}

Let $\mathbb D(\X)$ be the commutative algebra of $\G$-invariant differential operators on $\X$ and $S(\mathfrak{a}_\C)^\Wg$ the commutative algebra of $\Wg$-invariant polynomial functions on $\mathfrak{a}_\C^*$. 
The Harish-Chandra isomorphism is a  surjective isomorphism $\gamma:\mathbb{D}(\X) \to S(\mathfrak{a}_\C)^\Wg$ such that
$\gamma(\Delta)(\lambda)=\inner{\rho}{\rho}-\inner{\lambda}{\lambda}$.
See e.g. \cite[Ch. II, Theorems 4.3 and 5.18, and p. 299]{He2}.

Let $\lambda\in\mathfrak{a}_\C^*$. The joint eigenspace $\mathcal E_\lambda(\X)$ for
the algebra $\mathbb{D}(\X)$ is
\begin{equation}\label{eq:El}
\mathcal E_\lambda(\X)=\{f\in C^\infty(\X): \text{$Df=\gamma(D)(\lambda)f$
for all $D \in \mathbb D(\X)$}\}\,.
\end{equation}
See \cite[Ch. II, \S 2, no. 3 and Ch. III, \S 6]{He3}.
The group $\G$ acts on $\mathcal E_\lambda(\X)$ by the left regular representation:
\begin{equation}\label{eq:eigenrep}
[T_\lambda(g)f](x)=f(g^{-1}x)\qquad (g\in \G, \, x\in \X).
\end{equation}
Notice that $\mathcal E_\lambda(\X)=\mathcal E_{w\lambda}(\X)$
(and hence $T_\lambda=T_{w\lambda}$) for all $w\in \Wg$.
The subspace of $\K$-fixed elements in $\mathcal E_\lambda(\X)$ is 1-dimensional and spanned
by Harish-Chandra's spherical function $\varphi_\lambda$. The closed subspace
$\mathcal E_{(\lambda)}(\X)$ generated by the translates $T_\lambda(g)\varphi_\lambda$ of $\varphi_\lambda$, with $g \in \G$, is the unique closed irreducible subspace of
$\mathcal E_\lambda(\X)$. The restriction of $T_\lambda$ to $\mathcal E_{(\lambda)}(\X)$ is quasisimple and admissible. See e.g. \cite[Ch. IV, Theorem 4.5]{He2}.
Furthermore, the representation $T_\lambda$ is irreducible
if and only if $1/\Gamma_\X(\lambda)\neq 0$
where $\Gamma_\X$ is the Gamma function attached to $\X$, as in
\cite[Ch. III, \S 7, Theorem 6.2]{He3}. For $\X=\SL(3,\R)/\SO(3)$ we have
\begin{equation}
\label{eq:gammaX-SL3}
\Gamma_\X(\lambda)=
\prod_{\alpha\in \Sigma} \Gamma\big(\tfrac{3}{4}+\tfrac{\lambda_\alpha}{2}\big)
\Gamma\big(\tfrac{1}{4}+\tfrac{\lambda_\alpha}{2}\big)
=\prod_{\alpha\in \Sigma^+} \frac{2\pi^2}{\cos(\pi \lambda_\alpha)}\,.
\end{equation}
Thus $T_\lambda$ is reducible if and only if there is $\alpha \in \Sigma$ so that
$i\lambda_\alpha\in i\big(\mathbb{Z}+\frac{1}{2}\big)$.
In the present case, this is equivalent to $\lambda$ being a singularity of the Plancherel density
$[\cHC(\compl\lambda)\cHC(-\compl\lambda)]^{-1}$.
Here $\cHC(\l)$ denotes Harish-Chandra's $c$-function; see e.g. \cite[Theorem 4.7.5]{GV}. For $\SL(3,\R)$,
the Plancherel density is a meromorphic function on $\frak a_\C^*$, given by
\begin{equation}
\label{eq:Plancherel1}
[\cHC(\compl\lambda)\cHC(-\compl\lambda)]^{-1}=c_0 \prod_{\alpha \in \Sigma^+} \lambda_\alpha \th (\pi \lambda_\alpha)
\end{equation}
where $c_0$ is a normalizing constant and for $\lambda\in\fraka^*_\C$ and $\alpha\in\Sigma$
we have set
\begin{equation}
\label{eq:la}
\lambda_\alpha=\tfrac{\inner{\lambda}{\alpha}}{\inner{\alpha}{\alpha}}\,.
\end{equation}

The space $\mathcal E_{\lambda,\G}(\X)$ of $\G$-finite elements in
$\mathcal E_{\lambda}(\X)$ is a
(possibly zero) invariant subspace of $\mathcal E_{\lambda}(\X)$. Recall that it consists
of all functions $f \in \mathcal E_{\lambda}(\X)$ such that the vector space spanned by the left translates $T_\lambda(g)f$ of $f$, with $g \in \G$, is finite dimensional.

The following proposition holds for arbitrary Riemannian symmetric spaces of the noncompact type $\X=\G/\K$, where $\G$ is a noncompact connected semisimple Lie group with finite center and $\K$ a maximal compact subgroup of $\G$. It characterizes the $\mathcal E_{\lambda,\G}(\X)$, when non-zero, as finite-dimensional spherical representations of $\G$.

\begin{pro}
\label{prop:sphrep}
$\mathcal E_{\lambda,\G}(\X)\neq \{0\}$ if and only if there is $w\in \Wg$ so that
\begin{equation}
\label{eq:sphrep}
-(w\lambda)_\alpha-\rho_\alpha \in \mathbb{Z}^+ \qquad\text{for all $\alpha \in \Sigma^+$}\,.
\end{equation}
In this case, $\mathcal E_{\lambda,\G}(\X)$ is finite dimensional and
irreducible under $\G$. It is the finite-dimensional spherical
representation of highest restricted weight $-w\lambda-\rho$. All
finite dimensional spherical representations of $\G$ arise in this
fashion.
\end{pro}
\begin{proof}
This is \cite[Ch. II, Proposition 4.16]{He3}.
\end{proof}

\begin{rem}\label{rem:hw}
Notice that (\ref{eq:sphrep}) implies that $\lambda \in \fraka^*$.
Recall that a finite-dimensional representation $\pi$ is spherical if and only if its contregradient $\check{\pi}$ is spherical. Moreover, if $\pi$ has highest restricted weight $-w\lambda-\rho$, then $\check{\pi}$ has highest restricted weight $-w_0(-w\lambda-\rho)=w_0w\lambda-\rho$, where $w_0$ denotes the longest Weyl group element.
See e.g. \cite[Corollary 4.13 and Theorem 4.12]{He3}.
Thus $\mathcal E_{\lambda,\G}\neq \{0\}$ if and only if $\mathcal E_{-w_0\lambda,\G}\neq \{0\}$, which is equivalent to the existence of $w\in \Wg$ so that $(w\lambda)_\alpha-\rho_\alpha \in \mathbb{N}$ for all $\alpha \in \Sigma^+$. For a dominant $\lambda \in \fraka^*$, this condition is therefore equivalent to  $\lambda_\alpha-\rho_\alpha \in \mathbb{N}$ for all $\alpha \in \Sigma^+$.  In turn, for $\G=\SL(3,\R)$, the latter condition is equivalent to $\lambda_{\alpha}\in \mathbb{N} +\frac{1}{2}$ for $\alpha\in\{\alpha_{1,2},\alpha_{2,3}\}$
when $\lambda \in \fraka^*$ is dominant.
\end{rem}

\subsection{The resolvent of the Laplacian}
\label{subsection:resolvent}

Let $\Delta$ be the nonnegative Laplacian on $\X$ for the $\G$-invariant Riemannian structure associated
with the Cartan-Killing form of $\mathfrak{g}=\mathfrak{sl}(3,\R)$.  Then
$\Delta$ is an essentially self-adjoint $\G$-invariant differential operator on $L^2(\X)$ and
the spectrum of $\Delta$ is the half-line $[\rhoXtwo,+\infty)$ where
$\rhoXtwo=\inner{\rho}{\rho}=12$. The resolvent $R_\Delta(u)=(\Delta-u)^{-1}$
is therefore a holomorphic function of  $u \in \C \setminus [\rhoXtwo, +\infty)$.
It is actually convenient to consider the change of variable $z^2=u-\rhoXtwo$,
and reduce the study of $R_\Delta$ to that of
\begin{equation}
\label{eq:defR}
R(z)=(\Delta-\rhoXtwo-z^2)^{-1}=R_\Delta(\rhoXtwo+z^2).
\end{equation}
Here we are choosing the single-valued holomorphic branch of the square root function on $\C \setminus [0,+\infty)$ mapping  $-1$ to $i$. (Later the notation $\sqrt{\cdot}$ will be reserved to a different choice of holomorphic branch; see \eqref{square root 1}.) Hence $R$ is a holomorphic function of $z \in \C^+$, where $\C^+:=\{w \in \C: \Im w > 0\}$ is the upper half-space, with values in the space of bounded linear operators on $L^2(\X)$.

An explicit formula for $R(z)$ can be obtained by means of the Plancherel Theorem for the Helgason-Fourier transform; see e.g. \cite[Section 1.4]{HP09}. It follows, in particular, that for every $z \in \C^+$ and $f \in C_c^\infty(\X)$, the distribution $R(z)f$ is in fact a $C^\infty$ function on $\X$, given by the formula
\begin{equation}
\label{eq:resolventz}
[R(z)f](y)=
\frac{1}{|\Wg|}\,\int_{\mathfrak{a}^*} \frac{1}{\inner{\lambda}{\lambda}-z^2}\; (f \times \varphi_{\compl\lambda})(y) \; \frac{d\lambda}{\cHC(\compl\lambda)\cHC(-\compl\lambda)}\,,
\qquad y \in \X\,.
\end{equation}
The symbol $\times$ in (\ref{eq:resolventz}) denotes the convolution on $\X$. Recall that for
sufficiently regular functions $f_1, f_2:\X \to \C$, the convolution $f_1\times f_2$ is the function on $\X$ defined by $(f_1 \times f_2) \circ \pi= (f_1 \circ \pi) * (f_2 \circ \pi)$. Here $\pi:\G\to \X=\G/\K$ is the natural projection and $*$ denotes the convolution product of functions on $\G$.

The convolution $(f\times \varphi_{\compl\lambda})(y)$ can be described in terms of the Helgason-Fourier transform of $f$.
In fact, for $ f\in C_c^\infty(\X)$, $\lambda \in \mathfrak{a}_\C^*$ and $y=g\cdot o\in \X$, we have by
\cite[Ch. III, Lemma 1.2 and proof of Theorem 1.3]{He3} that
\begin{equation}\label{eq:conv-with-phil}
(f \times \varphi_\lambda)(y)=\int_\B \mathcal F f(\lambda,b)\,e_{\lambda,b}(y)\; db
\end{equation}
is the spherical Fourier transform of the $\K$-invariant function $f_y\in C_c^\infty(\X)$
given by
$$f_y(g_1)=\int_\K f(gkg_1 \cdot o) \; dk\,.$$
It follows by the Paley-Wiener Theorem that for every fixed $y\in \X$ the function $(f \times \varphi_{\compl\lambda})(y)$ is a Weyl-group-invariant entire function of $\lambda\in\mathfrak{a}_\C^*$ and there exists a constant $R \geq 0$ (depending on $y$ and on the size of the support of $f$) so that for each
$N \in \mathbb N$
\begin{equation}\label{eq:exptype-conv}
\sup_{\lambda\in\mathfrak{a}_\C^*} e^{-R|{\rm Im} \lambda|}(1+|\lambda|)^N
|(f\times \varphi_{\compl\lambda})(y)| < \infty\,.
\end{equation}

\section{An initial holomorphic extension of the resolvent}
We keep the notation introduced in the previous section. Recall in particular that the bottom of the spectrum of the Laplacian
is $\rhoXtwo=\inner{\rho}{\rho}=\inner{\alpha}{\alpha}$ for all $\alpha \in \Sigma$. In the following we denote by $\rhoX$ the positive square root of
$\rhoXtwo$.

Let $\{e_1,e_2\}$ be the standard basis in $\R^2$. We identify $\fraka^*$ with $\R^2$
by
\begin{equation}
\label{eq:alpha-e}
\alpha_{1,2}=e_1
\qquad \textrm{and} \qquad
\alpha_{2,3}=-\tfrac{1}{2} e_1 +\tfrac{\sqrt{3}}{2}e_2\,.
\end{equation}
In this way the inner product $\inner{\cdot}{\cdot}$ is $\rhoXtwo$-times the usual Euclidean inner product on $\R^2$.
One deduces from (\ref{eq:omega}) that
\begin{eqnarray}
\label{eq:omega-e}
&&w_{1,2}=e_1+ \tfrac{\sqrt{3}}{3} e_2
\qquad \textrm{and} \qquad
w_{2,3}=2\tfrac{\sqrt{3}}{3}e_2\,.
\end{eqnarray}
The above identification yields for
\begin{equation}
\label{eq:lambda-ident}
\lambda=\lambda_{1,2} w_{1,2}+ \lambda_{2,3} w_{2,3}
= x_1e_1+x_2e_2= x_1+ix_2\in \fraka^*\equiv \R^2 \equiv \C
\end{equation}
the relations:
\begin{eqnarray}
\label{eq:lambda-x}
\l_{1,2}&=&x_1\,, \notag \\
\l_{2,3}&=&-\tfrac{1}{2} x_1+ \tfrac{\sqrt{3}}{2} x_2\,,\\
\l_{1,3}&=&\l_{1,2}+\l_{2,3}=\tfrac{1}{2} x_1+ \tfrac{\sqrt{3}}{2} x_2\,.
\notag
\end{eqnarray}
Set
\begin{equation}
\label{eq:bxi}
\bxi=e^{i\frac{\pi}{3}}\,.
\end{equation}
In the above coordinates, one can rewrite (\ref{eq:Plancherel1}) as
\begin{eqnarray}
\label{eq:c-coords-x}
[\cHC(\compl\lambda)\cHC(-\compl\lambda)]^{-1}&=&c_0\,x_1 \th(\pi x_1)  (x_1\tfrac{1}{2}+x_2\tfrac{\sqrt{3}}{2}) \th(\pi (x_1\tfrac{1}{2}+x_2\tfrac{\sqrt{3}}{2}))\, \notag\\
&&\qquad \qquad \qquad \times (-x_1\tfrac{1}{2}+x_2\tfrac{\sqrt{3}}{2})\,\th(\pi (-x_1\tfrac{1}{2}+x_2\tfrac{\sqrt{3}}{2})) \notag\\
&=&c_0 \, \prod_{u\in\{1,\bxi,\bxi^2\}} \Re(\l u) \th\big(\pi \Re(\l u)\big)\,.
\end{eqnarray}

Let $\Sg^1=\{w \in \C; |w|=1\}$ be the unit circle and let
$\sigma$ denote the rotation-invariant probability measure on $\Sg^1$:
$$\int_{\Sg^1} f(w) \, d\sigma(w)=\frac{1}{2\pi} \int_0^{2\pi} f(e^{i\theta})\, d\theta=
\frac{1}{2\pi} \int_{\Sg^1} f(w) \; \frac{dw}{iw}\,.$$
Introduce polar coordinates in $\mathfrak{a}^*\equiv \R^2$: for $\lambda$
as in (\ref{eq:lambda-ident}), set
\begin{equation}
\label{eq:lpolar}
\lambda=\lrw=rw
\quad
\text{where} \quad
\begin{cases}
&r^2=x_1^2+x_2^2\,, \\
&w=\frac{\lambda}
{|\lambda|} \in \Sg^1\,.
\end{cases}
\end{equation}
Hence $\inner{\lrw}{\lrw}=\rhoXtwo r^2$.
Moreover, (\ref{eq:resolventz}) becomes
\begin{equation}
\label{eq:Rcalz-polar}
[R(z)f](y)=
\frac{1}{|\Wg|}\,\int_0^\infty \frac{1}{\rhoXtwo r^2-z^2} \; F(f,r,y) r \; dr
\end{equation}
where
\begin{eqnarray}
F(f,y,r)&=&\int_{\Sg^1} (f\times \varphi_{\compl \lrw})(y) \;
\frac{d\sigma(w)}{\cHC(\compl \lrw)\cHC(-\compl \lrw)}  \notag \\
\label{eq:F}
&=&\frac{1}{2\pi} \int_{\Sg^1} (f\times \varphi_{\compl \lrw})(y) \;
\frac{1}{\cHC(\compl \lrw)\cHC(-\compl \lrw)} \; \frac{dw}{iw}\,.
\end{eqnarray}
In the following, we will omit the dependence on $f\in C^\infty_c(\X)$ and $y\in \X$ from the notation,
and write $F(r)$ instead of $F(f,r,y)$, and $R(z)$ instead of $[R(z)f](y)$.

Set
\begin{eqnarray}
\label{eq:cos}
\cz(z)&=&\frac{z+z^{-1}}{2} \qquad  \qquad (z\in\C^\times)\,,\\
\label{eq:sin}
\sz(z)&=&\frac{z-z^{-1}}{2}  \qquad  \qquad (z\in\C^\times)\,.
\end{eqnarray}
If $z \in \Sg^1$, then  $\cz(z)=\Re z$ and $\sz(z)=i\Im z$.
Hence (\ref{eq:c-coords-x}) gives for $r>0$ and $w\in \Sg^1$:
\begin{eqnarray}
\label{eq:c-polar-coords}
[\cHC(\compl \lrw)\cHC(-\compl \lrw)]^{-1}&=  \notag
&c_0 \, \prod_{u\in\{1,\bxi,\bxi^2\}} \Re(ruw) \th\big(\pi \Re(ruw)\big)\\
&=&c_0 r^3 \, \prod_{u\in\{1,\bxi,\bxi^2\}} \cz(uw) \th\big(\pi r \cz(uw)\big)\,.
\end{eqnarray}

We extend the function $(r,w)\in [0,+\infty) \times \Sg^1 \to \l(r,w)=rw\in \fraka^*=\R^2$ by $\C$ linearity in $r$ to a map
from $\C \times \Sg^1$ to $\fraka_\C^*$ by setting
\begin{equation}
\label{eq:lzw}
\lzw=\l(x,w)+\compl \l(y,w) \qquad (z=x+iy\in \C, w\in \Sg^1)\,,
\end{equation}
where $\l(x,w)=xw=\frac{x}{|x|}\l(|x|,w)$ for $(x,w)\in (-\infty,0)\times \Sg^1$.

\begin{pro}
\label{pro:holoextRF}
Let $f\in C^\infty_c(\X)$ and $y\in \X$ be fixed.
\begin{enumerate}
\thmlist
\item The function $F(z)$, defined by (\ref{eq:F}) for $z=r>0$, extends to an even
holomorphic function  of  $z\in \C\setminus i((-\infty, -\frac{1}{2}]\cup[\frac{1}{2},\infty))$.
Moreover, $F(z)/z^6$ is bounded near $z=0$.
\item
Fix $x_0>0$ and $y_0>0$. Let
\begin{eqnarray*}
Q&=&\{z\in \C; \Re z>x_0,\ y_0> \Im z \geq 0\}\\
U&=&Q\cup  \{z\in \C; \Im z <0\}
\end{eqnarray*}
Then there is a holomorphic function $H:U\to \C$ such that
\begin{equation}
\label{eq:holoextRF}
R(\rhoX z)=H(z)+\frac{\pi i}{|W| \rhoXtwo}\, F(z) \qquad (z\in Q).
\end{equation}
In particular, the function $R(z)=[R(z)f](y)$ extends holomorphically from $\C^+$ to
$\C \setminus \big((-\infty,0] \cup i(-\infty, -\frac{1}{2}\rhoX]\big)$.
\end{enumerate}
\end{pro}
\begin{prf}
According to (\ref{eq:c-polar-coords}), the integrand in the definition (\ref{eq:F}) of $F$ is
\begin{equation}
\label{eq:integrandF}
\frac{1}{2\pi i w} (f\times \varphi_{\compl \lzw})(y) z^3 \, \prod_{u\in\{1,\bxi,\bxi^2\}} \cz(uw) \th\big(\pi z \cz(uw)\big)\,,
\end{equation}
which is a meromorphic function of $z\in \C$. Since $\th(\pi v)$ is holomorphic on $v\in \C\setminus i(\Bbb Z+\frac{1}{2})$ and $|\cz(w)|\leq 1$ for all $w \in \Sg^1$, the function $F$ extends to a holomorphic function of $z\in \C\setminus i((-\infty, -\frac{1}{2}]\cup[\frac{1}{2},+\infty))$. The change of variable given by the antipodal map $w\mapsto -w$ of $\Sg^1$ shows that $F$ is even. The last part of (a) follows from (\ref{eq:integrandF}) and the Dominated Convergence Theorem.

Since,
\[
\frac{2r}{r^2-z^2}=\frac{1}{r-z}+\frac{1}{r+z},
\]
we have
\begin{eqnarray} \label{original integral}
2 |W| \rhoXtwo \, R(\rhoX z)&=&2 \int_0^\infty \frac{F(r)}{r^2-z^2}\,r\,dr \notag\\
&=&\int_0^\infty \frac{F(r)}{r-z}\,dr+\int_0^\infty \frac{F(r)}{r+z}\,dr.
\end{eqnarray}
The functions
\begin{equation}\label{first integral}
\C\setminus [0,+\infty)\ni z\mapsto\int_0^\infty \frac{F(r)}{r-z}\,dr\in\C
\end{equation}
\begin{equation}\label{second integral}
\C\setminus (-\infty,0]\ni z\mapsto\int_0^\infty \frac{F(r)}{r+z}\,dr\in\C
\end{equation}
are holomorphic.

Let $z_0=x_0+iy_0$. Fix in the first quadrant a curve $\gamma_+$ that starts at $0$, goes to the right and up above $z_0$ and then becomes parallel to the positive real line and goes to infinity. We suppose that
$Q$ is interior to the region bounded by $\gamma_+$ and the positive real line.

Let $M, m$ be fixed positive numbers. Notice that, by (\ref{eq:exptype-conv}), the function $(f \times \varphi_{\compl \lzw})(y)$ is rapidly decreasing in the strip $\{z\in \C; |\Im z|\leq M\}$. Moreover $\th (\pi v)$ is bounded in the domain $\{v\in \C; |v|\leq 1/3\} \cup
\{v \in \C; \Re v \geq m\}$.
Let $z\in Q$. Then, by Cauchy's theorem,
\begin{equation}\label{first integral computed}
\int_0^\infty \frac{F(r)}{r-z}\,dr=\int_{\gamma_+} \frac{F(\zeta)}{\zeta-z}\,d\zeta + 2\pi i F(z)\,.
\end{equation}
The term
\[
\int_{\gamma_+} \frac{F(\zeta)}{\zeta-z}\,d\zeta
\]
has a holomorphic extension to the forth and third quadrant through the positive reals and the negative imaginary numbers. Moreover,
\[
\int_0^\infty \frac{F(r)}{r^2-z^2}\,r\,dr=\frac{1}{2}\left(\int_0^\infty \frac{F(r)}{r+z}\,dr+\int_{\gamma_+} \frac{F(\zeta)}{\zeta-z}\,d\zeta\right)+\pi i F(z).
\]
\end{prf}

\begin{rem}
Write the elements of $\Sg^1$ as $w=e^{i\theta}$. So $\cz(w)=\cos\theta$.
Since
\[
 \cos\theta \th(\tfrac{\pi i}{2}\cos\theta)=-\pi \cos\theta\, \frac{\sin(\frac{\pi}{2}\cos\theta)}{\cos(\frac{\pi}{2}\cos\theta)}
\]
and since for $\theta$ near $0$ the above expression is approximately equal to
\[
-\pi \frac{1}{\frac{\pi}{2}(1-\cos\theta)}
=- \frac{2}{(1-\cos\theta)}
\approx -\frac{4}{\theta^2},
\]
the function (\ref{eq:integrandF})  is not absolutely integrable if $z=\tfrac{i}{2}$. Therefore it is not clear if  $F$ extends from the right or left half plane to a continuous function on any set  containing parts of the imaginary axis larger than $(-\frac{i}{2},\frac{i}{2})$.
\end{rem}

The extension of $R(z)$ across $(-\infty,0]$ can be deduced from the results of Mazzeo and Vasy \cite{MV05} and of Strohmaier \cite{Str05}. They require an additional change of variables.

Let $\log$ denote the holomorphic branch of the logarithm defined on
$\C \setminus ]-\infty, 0]$ by $\log 1=0$. Set $\tau=\log z$.
It gives a biholomorphism between $\C^+$ and the strip $S_{0,\pi}:=\{\tau \in \C: 0<\Im \tau < \pi\}$.
Let $f\in C^\infty_c(\X)$ and $y\in \X$ be fixed. Define
\begin{equation}
\label{eq:R-log}
[\Rlog\null(\tau)f](y)=[R(\rhoX e^{\tau})f](y)=
\frac{1}{|W|}\,\int_{\fraka^*} \frac{1}{\inner{\l}{\l}-\rhoXtwo e^{2\tau}}\; (f \times \varphi_{\compl\l})(y) \; \frac{d\l}{\cHC(\compl\l)\cHC(-\compl\l)}\,.
\end{equation}
Polar coordinates in $\fraka^*$ give now
\begin{equation}
\label{eq:R-log-polar}
[\Rlog(\tau)f](y)=
\frac{1}{\rhoXtwo |\Wg|}\,\int_{-\infty}^{+\infty} \frac{1}{e^{2t}-e^{2\tau}}\; F(e^t) e^{2t} \; dt\,,
\end{equation}
where $F(r)=F(f,y,r)$ is as in (\ref{eq:F}).
The evenness of the function $F$ becomes $i\pi$-periodicity of the function $F(e^t)$.

As for the functions $F$ and $R$, we will omit the dependence of $f$ and $y$ in the notation
and write $\Rlog(\tau)$ instead of $[\Rlog(\tau)f](y)$.

\begin{pro}
\label{pro:holoextRFnegaxis}
Let $f\in C^\infty_c(\X)$ and $y\in \X$ be fixed. The function $\Rlog(\tau)=[\Rlog(\tau)f](y)$ extends
holomorphically from $S_{0,\pi}$ to the open set
$$\Ulog=\C \setminus \bigcup_{n \in \mathbb Z\setminus \{0\}}
\Big(i\pi \big(n+\tfrac{1}{2}\big)+[\log(\tfrac{1}{2}),+\infty)\Big)$$
and satisfies the identity:
\begin{equation}
\Rlog(\tau+i\pi)=\Rlog(\tau)+\frac{i\pi}{\rhoXtwo |\Wg|}\, F(e^\tau)\qquad (\tau \in \Ulog)\,.
\end{equation}
Consequently,
\begin{multline}
[R(\rhoX z e^{i\pi})f](y)=[R(\rhoX z)f](y)+\frac{i\pi}{\rhoXtwo |\Wg|}\, F(z)\\ \qquad (z \in \C \setminus \big(i(-\infty, -\tfrac{1}{2}]\cup i[\tfrac{1}{2},+\infty) \cup (-\infty,0]\big))\,.
\end{multline}
\end{pro}
\begin{prf}
This is \cite[Proposition 4.3]{Str05} with $f(x)=F(x)x$ for $x \in [0,+\infty)$; see also \cite[Theorem 1.3]{MV05}.
\end{prf}

Since $F(0)=0$, we obtain the following corollary.
\begin{cor}
\label{cor:logRiemann}
Let $f\in C^\infty_c(\X)$ and $y\in \X$ be fixed. Then the function $R(z)=[R(z)f](y)$ satisfies
\begin{multline}
R(z e^{2i\pi})=R(z)+\frac{2 i\pi}{\rhoXtwo |\Wg|}\, F(\rhoX^{-1} z)\\ \qquad (z \in \C \setminus  \big((-\infty,0]\cup i(-\infty, -\tfrac{1}{2}\rhoX]\cup i[\tfrac{1}{2}\rhoX,+\infty) \big))\,.
\end{multline}
and extends holomorphically from
$\C \setminus \big((-\infty,0] \cup i(-\infty, -\frac{1}{2}\rhoX]\big)$ to a logarithmic Riemann
surface branched along $(-\infty,0]$, with the preimages of  $i\big((-\infty, -\tfrac{1}{2}\rhoX]\cup [\tfrac{1}{2}\rhoX, +\infty)\big)$ removed.
\end{cor}

Proposition \ref{pro:holoextRF} and Corollary \ref{cor:logRiemann} show that all possible resonances of the resolvent $R(z)$ of the Laplacian of $\X$ are located along the half-line $i(-\infty, -\tfrac{1}{2}\rhoX]$.
Because of (\ref{eq:holoextRF}), the study of $R(z)$ along $i(-\infty, -\frac{1}{2}\rhoX]$
is reduced to that of  $F(z)$ along $i(-\infty, -\frac{1}{2}]$. This will be done in the following section.

\section{Meromorphic extension}
\label{section:mero-ext}

In this section we complete the meromorphic extension of the resolvent of the Laplacian. We first need a different expression for the function
$F$ from (\ref{eq:F}).

Recall that, for fixed $f \in C_c^\infty(\X)$ and $y \in \X$,
\begin{equation*}
F(z)=F(f,y,z)=
\frac{1}{2\pi} \, \int_{\Sg^1} (f\times \varphi_{\compl \lzw})(y) z^3 \, \prod_{u\in\{1,\bxi,\bxi^2\}} \cz(uw) \th\big(\pi z \cz(uw)\big)\; \frac{dw}{iw}\,,
\end{equation*}
where $\bxi=e^{i\pi/3}$ and $\cz(z)$ is as in (\ref{eq:cos}).
Observe that
\begin{equation*}
F(z)=
\frac{1}{2\pi} \, \int_{\Sg^1} \Big(f\times \frac{\varphi_{\compl \lzw}+\varphi_{-\compl \lzw}}{2} \Big)(y)  z^3 \, \prod_{u\in\{1,\bxi,\bxi^2\}} \cz(uw) \th\big(\pi z \cz(uw)\big)\; \frac{dw}{iw}\,,
\end{equation*}
where now $\big(f\times \tfrac{\varphi_{\compl \lzw}+\varphi_{-\compl \lzw}}{2} \big)(y)$ is an
even function of $z$ and $w$.

Let
\begin{eqnarray}\label{phiwu}
\phi_{z,u}(w)&=&\frac{z \cz(uw)}{iw}\,\th(\pi z \cz(uw))\qquad (z\cz(uw)\notin i(\Bbb Z+\tfrac{1}{2}),\ z\in \C,\ w\in \C^\times),\\
\label{psiwu}
\psi_z(w)&=&\frac{1}{2\pi}\,\big(f\times \tfrac{\varphi_{\compl \lzw}+\varphi_{-\compl \lzw}}{2} \big)(y)(iw)^2\qquad (z\in \C,\ w\in \C^\times).
\end{eqnarray}
Hence
\begin{equation}
\label{eq:Fbis}
F(z)=\int_{\Sg^1} \psi_w(z) \prod_{u\in\{1,\bxi,\bxi^2\}} \phi_{z,u}(w) \; dw\,.
\end{equation}
Note that
\begin{equation}\label{even phiwu}
\phi_{z,u}(w)=\phi_{\pm z,\pm u}(w)\qquad \text{and}\qquad\psi_z(w)=\psi_{-z}(w).
\end{equation}
Since $\l\mapsto \varphi_{\compl \l}+\varphi_{-\compl \l}$ is Weyl group invariant and even, it is also invariant under rotations by multiples of $\pi/3$. It follows that for $u=e^{i\theta}$ with $\theta \in \big\{\pm\tfrac{\pi}{3}, \pm\tfrac{2\pi}{3}\big\}$ we have
\begin{equation}
\label{eq:rotation-inv-pithirds}
\psi_z(uw)=\psi_z(w) \, u^2\,.
\end{equation}
Moreover
\begin{equation}\label{even phiwu bis}
\phi_{z,u}(-w)=-\phi_{z,u}(w)\qquad \text{and}\qquad \psi_z(-w)=\psi_{z}(w).
\end{equation}

\subsection{Analytic properties of the function $\cz(z)$}
To proceed further, we need to recall some known analytic properties of the functions $\cz(z)$ and $\sz(z)$. Recall that for $z \in \C^\times$ these functions are defined by
\begin{equation*}
\cz(z)=\frac{z+z^{-1}}{2}\quad \text{and} \quad
\sz(z)=\frac{z-z^{-1}}{2} \,.
\end{equation*}
For $r>0$ and $a,b \in \R\setminus \{0\}$ let
\begin{eqnarray*}
\Dg_r&=&\{z\in \C;\ |z|<r\}\,,\\
\Eg_{a,b}&=&\{\xi+i\eta\in \C;\ \left(\frac{\xi}{a}\right)^2+ \left(\frac{\eta}{b}\right)^2<1\}\,.
\end{eqnarray*}
Their boundaries $\partial \Dg_r$ and $\partial \Eg_{a,b}$ are respectively the circle of radius $r$ and the ellipse of semi-axes $|a|,|b|$, both centered at $0$. In particular, $\partial \Dg_1=\Sg^1$.
Furthermore, $\overline{\Dg_r}$ is the closed disc of center $0$ and radius $r$.

Observe that $\cz:\Dg_1\setminus \{0\} \to \C\setminus [-1,1]$ is a biholomorphic function.
We denote by $\cz^{-1}$ its inverse.
More precisely, for $0<r<1$, it restricts to a biholomorphic function
\begin{equation}\label{first bijection}
\cz:\Dg_1\setminus\overline{\Dg_r}\to  \Eg_{\cz(r), \sz(r)}\setminus [-1,1]\,.
\end{equation}
Also,
\begin{equation}\label{second bijection}
\cz:\partial \Dg_r\to \partial  \Eg_{\cz(r), \sz(r)}.
\end{equation}
is a bijection. In particular,
\[
Re^{i\Theta}\cz(\partial \Dg_r)=Re^{i\Theta}\partial  \Eg_{\cz(r), \sz(r)}
\]
is the ellipse $\partial  \Eg_{\cz(r), \sz(r)}$ rotated by $\Theta$ and dilated by $R$.

Let $\sqrt{\cdot}$ denote the single-valued holomorphic branch of the square root function defined on $\C \setminus (-\infty,0]$ by
\begin{equation}\label{square root 1}
\sqrt{Re^{i\Theta}}=\sqrt{R}\,e^{\frac{i\Theta}{2}} \qquad (R>0, \,-\pi<\Theta<\pi).
\end{equation}
A straightforward computation shows that
\begin{equation}\label{square root cartesian}
\sqrt{x+iy}=\sqrt{\frac{\sqrt{x^2+y^2}+x}{2}}+i\sign(y)\,\sqrt{\frac{\sqrt{x^2+y^2}-x}{2}}
\qquad (x+iy\in \C\setminus \R^-).
\end{equation}
\begin{lem}\label{two square roots}
The function $\sqrt{z+1}\sqrt{z-1}$ originally defined on $\C\setminus (-\infty, 1]$ extends to a holomorphic function on $\C\setminus [-1,1]$. Also,
$\sqrt{(-z)+1}\sqrt{(-z)-1}=-\sqrt{z+1}\sqrt{z-1}$.
\end{lem}
\begin{prf}
For $z=x+iy$ with $y\ne 0$, write
\begin{eqnarray*}
A=(x+1)^2+y^2\,, &&\qquad  a=x+1\,,\\
B=(x-1)^2+y^2\,, &&\qquad  b=x-1\,.
\end{eqnarray*}
Then
\begin{eqnarray*}
2\sqrt{z+1}\sqrt{z-1}&=&
\Big(\sqrt{A+a}+i\sign(y)\,\sqrt{A-a}\Big)
\Big(\sqrt{B+b}+i\sign(y)\,\sqrt{B-b}\Big)\\
&=&\Big(\sqrt{A+a}\sqrt{B+b}-\sqrt{A-a}\sqrt{B-b}\Big)\\
&&+i\ \sign(y)\,
\Big(\sqrt{A+a}\sqrt{B-b}+\sqrt{A-a}\sqrt{B+b}\Big)\,.
\end{eqnarray*}
Hence,
\begin{eqnarray*}
\underset{y\to 0\pm}{\lim}2\sqrt{z+1}\sqrt{z-1}&=&
\Big(\sqrt{|a|+a}\sqrt{|b|+b}-\sqrt{|a|-a}\sqrt{|b|-b}\Big)\, \\
& & \pm\ i\Big(\sqrt{|a|+a}\sqrt{|b|-b}+\sqrt{|a|-a}\sqrt{|b|+b}\Big)\,.
\end{eqnarray*}
If $x+1<0$, then $b<a<0$. So
\begin{equation}
\underset{y\to 0\pm}{\lim} \sqrt{z+1}\sqrt{z-1}=-\sqrt{|a||b|}=-\sqrt{x^2-1}
\end{equation}
is real. Thus $\sqrt{z+1}\sqrt{z-1}$ extends holomorphically across $(-\infty,-1)$ by Schwarz's reflection principle.

Checking the last formula is straightforward.
\end{prf}
\begin{lem}\label{scinverse}
For $z\in\C\setminus [-1,1]$, the following formulas hold:
\begin{eqnarray*}
&&(z-\sqrt{z+1}\sqrt{z-1})(z+\sqrt{z+1}\sqrt{z-1})=1,\\
&&\cz^{-1}(z)=z-\sqrt{z+1}\sqrt{z-1},\\
&&\sz\circ \cz^{-1}(z)=-\sqrt{z+1}\sqrt{z-1}.
\end{eqnarray*}
Also
\begin{eqnarray*}
&&\sqrt{x+i0+1}\sqrt{x+i0-1}= i \sqrt{1-x^2}\qquad (-1<x<1),\\
&&\sqrt{x-i0+1}\sqrt{x-i0-1}= -i \sqrt{1-x^2}\qquad (-1<x<1),\\
&&\cz^{-1}(x\pm i0)=x \mp i \sqrt{1-x^2}\qquad (-1<x<1).
\end{eqnarray*}
In particular, the ``upper part" of the interval $[-1,1]$ is mapped via $c^{-1}$ onto the lower half of the unit circle and  the ``lower part" of the interval $[-1,1]$ is mapped onto the upper half of the unit circle.
Furthermore
\begin{eqnarray*}
&&\sz\circ \cz^{-1}:\partial  \Eg_{\cz(r), \sz(r)}\to \partial  \Eg_{\sz(r), \sz(r)},\\
&&\sz\circ \cz^{-1}:\Eg_{\cz(r), \sz(r)}\setminus [-1,1]\to \Eg_{\sz(r), \cz(r)}\setminus i[-1,1],
\end{eqnarray*}
are bijections.
\end{lem}
\begin{prf}
The first equation is obvious. The two sides of the second equation are holomorphic and equal for $z>1$, hence equal in the whole domain of definition. Also,
\begin{eqnarray*}
\sz\circ \cz^{-1}(z)&=&\frac{z-\sqrt{z+1}\sqrt{z-1}-(z-\sqrt{z+1}\sqrt{z-1})^{-1}}{2}\\
&=&\frac{z-\sqrt{z+1}\sqrt{z-1}-(z+\sqrt{z+1}\sqrt{z-1})}{2}=-\sqrt{z+1}\sqrt{z-1} \qquad (0<r<1).
\end{eqnarray*}
The two limits are easy to compute.
The last statement follows from the fact that
\[
\sz:\partial \Dg_r\to \partial  \Eg_{\sz(r), \cz(r)}
\]
and
\[
\sz:\Dg_1\setminus\overline{\Dg_r}\to  \Eg_{\sz(r), \cz(r)}\setminus i[-1,1]
\]
are bijections.
\end{prf}
\begin{lem}\label{rays and circles}
Let $|\zeta|=|\zeta_0|=1$. Then
\begin{eqnarray}\label{rays and circles1}
\left(\zeta \cz^{-1}(\zeta_0 \R^+\setminus[-1,1])\right)\cap \left(\cz^{-1}(\zeta_0 \R^+\setminus[-1,1])\right)\ne\emptyset\ \ \text{if and only if}\ \ \zeta=1.
\end{eqnarray}
Also
\begin{eqnarray}\label{rays and circles2}
\left(\zeta \cz^{-1}(\zeta_0 \R\setminus[-1,1])\right)\cap \left(\cz^{-1}(\zeta_0 \R\setminus[-1,1])\right)\ne\emptyset\ \ \text{if and only if}\ \zeta=\pm 1.
\end{eqnarray}
\end{lem}
\begin{prf}
Recall that $\cz^{-1}$ is a bijection of $\C\setminus [-1,1]$ onto $\Dg_1\setminus \{0\}$.
If the intersection (\ref{rays and circles1}) is not empty then there is $0<r<1$ and $z$ with $|z|=r$ such that the two points $z$ and $\zeta z$ are mapped by $c$ to a point on the ray $\zeta_0 \R^+$. But the image of the circle $\partial \Dg_r$ is the ellipse $\partial \Eg_{\cz(r),\sz(r)}$, which has only one point of intersection with the ray. So $\cz(z)=\cz(\zeta z)$. Hence $z=\zeta z$, which means that $\zeta=1$. This verifies (\ref{rays and circles1}).

Since $\cz(-z)=-\cz(z)$ a similar argument proves (\ref{rays and circles2}).
\end{prf}

\subsection{Deformation of $\Sg^1$ and residues}
\label{subsection:deformationS1}
We keep the notation introduced above.
\begin{lem}\label{general formula}
Suppose $z\in \C\setminus i((-\infty, -\frac{1}{2}]\cup[\frac{1}{2},\infty))$ and $0<r<1$ are such that
\begin{eqnarray}\label{residue condition}
i(\Bbb Z+\tfrac{1}{2})\cap z\partial \Eg_{\cz(r),\sz(r)}=\emptyset.
\end{eqnarray}
Then
\begin{equation}\label{residue theorem}
F(z)=F_r(z)+2 \pi i \, G_r(z)\,,
\end{equation}
where
\begin{eqnarray}
\label{holomorphic part}
F_r(z)&=&\int_{\partial D_r}\psi_z(w)\prod_{u\in\{1,\bxi,\bxi^2\}}\phi_{z,u}(w)\,dw,\\
\label{eq:Gr}
G_r(z)&=&\displaystyle{\sum}_{w_0}' \,
\psi_z(w_0)\,\Res_{w=w_0}\,\prod_{u' \in\{1,\bxi,\bxi^2\}}\phi_{z,u'}(w)\,,
\end{eqnarray}
where $\displaystyle{\sum}_{w_0}'$ denotes the sum over the $w_0$'s for which there is a $u\in \{1,\bxi,\bxi^2\}$ so that
$$z\cz(uw_0)\in  i(\Bbb Z+\tfrac{1}{2})\cap z(\Eg_{\cz(r), \sz(r)}\setminus [-1,1])\,.$$
For a fixed $r$, both $F_r$ and $G_r$ are holomorphic on the open subset of $\C\setminus i((-\infty, -\frac{1}{2}]\cup[\frac{1}{2},\infty))$ where the condition (\ref{residue condition}) holds. Furthermore, $F_r$ extends to a holomorphic function on the open subset of $\C$ where the condition (\ref{residue condition}) holds. Also, $F_r(0)=G_r(0)=0$.
\end{lem}
\begin{prf}
Observe that, by (\ref{first bijection}), the condition that $w_0\in \Dg_1\setminus \overline{\Dg_r}$ and there is $u\in \{1,\bxi,\bxi^2\}$ so that $z\cz(uw_0)\in  i(\Bbb Z+\tfrac{1}{2})$ is equivalent to that there is $u\in \{1,\bxi,\bxi^2\}$ for which $z\cz(uw_0)\in  i(\Bbb Z+\tfrac{1}{2})\cap z(\Eg_{\cz(r), \sz(r)} \setminus [-1,1])$. Formula (\ref{residue theorem}) follows then from (\ref{eq:Fbis}) and the Residue Theorem.

Since $\phi_{z,u}(w)$ and $\psi_z(w)$ are holomorphic functions of the two variables $z$ and $w$, we see from (\ref{second bijection}) that $F_r$ is a holomorphic function of $z$ in the region where (\ref{residue condition}) holds. Also $F$ is holomorphic on $\C\setminus i((-\infty, -\frac{1}{2}]\cup[\frac{1}{2},\infty))$, hence $G_r$ is also holomorphic there.

Since $\phi_{0,u}=0$ we obtain that $F_r(0)=G_r(0)=0$.
\end{prf}
We now derive an explicit formula for the function $G_r$ in (\ref{eq:Gr}).

\begin{lem}\label{computation of residues}
Let  $z$ and $r$ be as in (\ref{residue condition}). Fix $u \in \{1, \bxi, \bxi^2\}$. Let $n_0\in\Bbb Z$ be such that $i(n_0+\frac{1}{2})\in  z(\Eg_{\cz(r), \sz(r)}\setminus [-1,1])$ and let $w_0\in \Dg_1\setminus\overline{\Dg_r}$ be such that $z\cz(uw_0)=i(n_0+\frac{1}{2})$. Then $z\ne 0$,
\begin{equation}\label{computation of residues0}
w_0=\frac{1}{u}\, \cz^{-1}\Big(\frac{i}{z}\big(n_0+\frac{1}{2}\big)\Big)
\end{equation}
and
\begin{equation}\label{computation of residues1}
\Res_{w=w_0}\prod_{u'\in \{1, \bxi, \bxi^2\}}\phi_{z,u'}(w)=\Big(\prod_{u'\ne u}\phi_{z,u'}(w_0)\Big)\cdot
\frac{\frac{i}{z}(n_0+\frac{1}{2})}{i\pi (\sz\circ \cz^{-1})(\frac{i}{z}(n_0+\frac{1}{2}))}.
\end{equation}
Conversely, suppose that $n_0\in\Bbb Z$ is so that $i(n_0+\frac{1}{2})\in  z(\Eg_{\cz(r), \sz(r)}\setminus [-1,1])$, and define $w_0$ by (\ref{computation of residues1}). Then
$w_0\in \Dg_1\setminus\overline{\Dg_r}$ and $z\cz(uw_0)=i(n_0+\frac{1}{2})$.
\end{lem}
\begin{prf}
Since $0\notin i(\Bbb Z+\frac{1}{2})$, we see that $z\ne 0$.
If $u'\ne u$, then Lemma \ref{rays and circles} shows that the set of the singularities of the function $\phi_{z,u'}$ within the punctured  open unit disc $\Dg_1\setminus\{0\}$ is disjoint from the set of the singularities of the function $\phi_{z,u}$. Therefore
\[
\Res_{w=w_0}\prod_{u'}\phi_{z,u'}(w)=\Big(\prod_{u'\ne u}\phi_{z,u'}(w_0)\Big)\cdot \Res_{w=w_0}\phi_{z,u}(w).
\]
Furthermore,
\begin{eqnarray*}
\Res_{w=w_0}\phi_{z,u}(w)&=&\frac{z\cz(uw_0)}{iw_0} \sh(\pi z \cz(uw_0)) \frac{1}{\frac{d}{dw} \ch(\pi w \cz(uz))|_{w=w_0}}\\
&=&\frac{z\cz(uw_0)}{iw_0} \sh(\pi z \cz(uw_0)) \frac{1}{\sh(\pi z \cz(uw_0))\pi z u \sz(uw_0)(uw_0)^{-1}}\\
&=&\frac{z\cz(uw_0)}{i\pi \sz(uw_0)}=\frac{\frac{i}{z}(n_0+\frac{1}{2})}{i\pi (\sz\circ \cz^{-1})(\frac{i}{z}(n_0+\frac{1}{2}))}.
\end{eqnarray*}
The last statement follows immediately from (\ref{first bijection}).
\end{prf}

\begin{lem}
\label{lemma:sumphipsi}
For all $z \in \C$ and $w \in \C^\times$ we have
\begin{eqnarray}\label{lemma:sumphipsi1}
\sum_{u\in \{1,\bxi,\bxi^2\}} \psi_z\big(\tfrac{1}{u} w\big) \Big(\prod_{u'\ne u}\phi_{z,u'}\big(\tfrac{1}{u}w\big)\Big)
=-3 \psi_z(w) \phi_{z,1}(\bxi w)\phi_{z,1}(\bxi^2 w)\,.
\end{eqnarray}
\end{lem}

\begin{proof}
Notice first that
\begin{equation*}
\phi_{z,u'}\big(\tfrac{1}{u}w\big)= \frac{z\cz\big(\tfrac{u'}{u}w\big)}{i\tfrac{1}{u}\,w} \, \th \big(\pi z \cz\big(\tfrac{u'}{u}w\big)\big)
=u' \frac{z\cz\big(\tfrac{u'}{u}w\big)}{i\tfrac{u'}{u}\,w} \, \th \big(\pi z
\cz\big(\tfrac{u'}{u}w\big)\big)
=u' \phi_{z,1}\big(\tfrac{u'}{u}w\big)\,.
\end{equation*}
By  (\ref{eq:rotation-inv-pithirds}), we have
$$\psi_z\big(\tfrac{1}{u} w\big)=\psi_z(w)\, \tfrac{1}{u^2}\,.$$
Hence
\begin{eqnarray}
\sum_u \psi_z\big(\tfrac{1}{u} w\big) \Big(\prod_{u'\ne u}\phi_{z,u'}\big(\tfrac{1}{u}w\big)\Big)
&=& \sum_u \psi_z(w) \tfrac{1}{u^2}  \Big(\prod_{u'\ne u} u' \phi_{z,1}\big(\tfrac{u'}{u}w\big)\Big) \nn\\
&=& \sum_u \psi_z(w)   \Big(\prod_{u'\ne u} \tfrac{u'}{u}\,
\phi_{z,1}\big(\tfrac{u'}{u}w\big)\Big)\,.
\label{eq:simplifyRes}
\end{eqnarray}
Since
\[
-\tfrac{u'}{u}\,  \phi_{z,1}\big(-\tfrac{u'}{u}w\big)=\tfrac{u'}{u}\,  \phi_{z,1}\big(\tfrac{u'}{u}w\big),
\]
we may assume that in the above product, for each fixed $u$,
\[
\tfrac{u'}{u} =\pm \frac{1}{2}+i \frac{\sqrt{3}}{2} \in \{\bxi,\bxi^2\}\,.
\]
Then
\[
\prod_{u'\ne u} \tfrac{u'}{u}\,  \phi_{z,1}\big(\tfrac{u'}{u}w\big)= \bxi \phi_{z,1}(\bxi w) \bxi^2\phi_{z,1}(\bxi^2 w)\,.
\]
Since $\bxi^3=-1$, we conclude that (\ref{eq:simplifyRes}) is equal to the right-hand side of (\ref{lemma:sumphipsi1}).
\end{proof}

\begin{cor}\label{computation of G(w)}
Under the assumptions of Lemma  \ref{general formula},
\begin{eqnarray}\label{computation of G(w)1}
G_r(z)&=&\sum_{n\in S_{r,z}}G_{(n)}(z),
\end{eqnarray}
where
\begin{equation}\label{boundedness3.1}
\Sg_{r,z}=\{n\in \Bbb Z;\ i\big(n+\tfrac{1}{2}\big)\in  z(\Eg_{\cz(r), \sz(r)}\setminus[-1,1])\}
\end{equation}
and
\begin{multline}\label{boundedness3}
G_{(n)}(z)=-3 \psi_z\big(\cz^{-1}(\tfrac{i}{z}(n+\tfrac{1}{2})\big)
\phi_{z,1}\big(\bxi \cz^{-1}(\tfrac{i}{z}(n+\tfrac{1}{2})\big)
\phi_{z,1}\big(\bxi^2 \cz^{-1}(\tfrac{i}{z}(n+\tfrac{1}{2})\big)\\
\times  \frac{\frac{i}{z}(n+\frac{1}{2})}{i\pi (\sz\circ \cz^{-1})(\frac{i}{z}(n+\frac{1}{2}))}\,.
\end{multline}
The function (\ref{boundedness3}) is holomorphic and even on $\C\setminus i\R$.
It satisfies
\begin{eqnarray}\label{boundedness3.6}
G_{(n)}(z)=G_{(-n-1)}(z).
\end{eqnarray}
Moreover, $n \in \Sg_{r,-z}$ if and only if $-n-1 \in \Sg_{r,z}$.
\end{cor}
\begin{prf}
The equality (\ref{computation of G(w)1}) is immediate from Lemmas \ref{computation of residues} and \ref{lemma:sumphipsi}.
If $z\in \C\setminus i\R$, then $\frac{i}{z}(n+\frac{1}{2})\in \C\setminus i\R$. Hence $G_{(n)}(z)$ is holomorphic on $\C\setminus i\R$. It is even, because of (\ref{even phiwu}) and (\ref{even phiwu bis}) and since $\cz^{-1}(-w)=-\cz^{-1}(w)$ for $w \in \C\setminus [-1,1]$.
The equality $G_{(n)}(-z)=G_{(-n-1)}(z)$ follows from Lemmas \ref{two square roots} and \ref{scinverse} together with (\ref{even phiwu}) and (\ref{even phiwu bis}). This proves
(\ref{boundedness3.6}) as $G_{(n)}(z)$ is even.
\end{prf}
\begin{cor}\label{boundedness}
Let $W\subseteq \C$ be a connected open set and let  $0<r<1$ be such that
\begin{eqnarray}\label{boundedness1}
i(\Bbb Z+\tfrac{1}{2})\cap W \partial \Eg_{\cz(r),\sz(r)}=\emptyset.
\end{eqnarray}
Then the set $\Sg_r:=\Sg_{r,z}$ defined in (\ref{boundedness3.1}) does not depend on $z\in W\setminus i\R$ and
\begin{eqnarray}\label{boundedness4}
G_r(z)=\sum _{n\in \Sg_r} G_{(n)}(z) \qquad (z\in W\setminus i\R).
\end{eqnarray}
Also,
\begin{eqnarray}\label{boundedness5.1}
n\in \Sg_r\ \ \text{if and only if}\ \ -n-1\in \Sg_r.
\end{eqnarray}
\end{cor}
\begin{prf}
The condition (\ref{boundedness1}) implies that for any fixed $n \in \mathbb Z$ the set
$\{z\in W: i(n+\tfrac{1}{2})\in  z \Eg_{\cz(r), \sz(r)}\}$ is open and closed in $W$.
Since $W$ is connected, this set is either $\emptyset$ or $W$. Hence
\begin{equation}\label{boundedness5}
i(\Bbb Z+\tfrac{1}{2})\cap  z \Eg_{\cz(r), \sz(r)}=i(\Bbb Z+\tfrac{1}{2})\cap  W \Eg_{\cz(r), \sz(r)}
\qquad (z\in W).
\end{equation}
Notice that if $z\in \C\setminus i\R$, then
\[
i(\Bbb Z+\tfrac{1}{2})\cap  z \Eg_{\cz(r), \sz(r)}=i(\Bbb Z+\tfrac{1}{2})\cap  \left(z(\Eg_{\cz(r), \sz(r)}\setminus[-1,1])\right).
\]
Hence, for each $z\in W\setminus i\R$
\begin{eqnarray}\label{boundedness5.00}
\Sg_{r,z} = i(\Bbb Z+\tfrac{1}{2})\cap  W \Eg_{\cz(r), \sz(r)}.
\end{eqnarray}
Therefore (\ref{boundedness4}) follow from Corollary \ref{computation of G(w)}.

Since our ellipses are invariant under the multiplication by $-1$, the formula (\ref{boundedness5.1}) follows from (\ref{boundedness3.1}).
\end{prf}
\begin{lem}\label{connected nbh}
For any $iv\in i\R$ there is a connected neighborhood $W\subseteq\C$ of $iv$ and $0<r<1$ such that the condition (\ref{boundedness1}) holds.
\end{lem}
\begin{prf}
Choose $0<r<1$ so that
\[
i(\Bbb Z+\tfrac{1}{2})\cap  iv \partial \Eg_{\cz(r), \sz(r)}=\emptyset.
\]
Then enlarge $iv$ to $W$.
\end{prf}

Lemma \ref{connected nbh} shows that for any point $iv\in i\R$ there is a neighborhood $W$ of that point and a radius $0<r<1$ so that condition (\ref{boundedness1}) is satisfied. Lemma \ref{general formula} and  Corollary \ref{boundedness} show that
\begin{eqnarray}\label{main cor1}
F(z)=F_r(z)+2\pi i\sum_{n\in \Sg_r} G_{(n)}(z) \qquad (z\in W\setminus i\R),
\end{eqnarray}
where $F_r$ extends to a holomorphic function on $W$. Let $N\geq 0$ be the largest element in $\Sg_r$. Then (\ref{boundedness3.1}) and (\ref{boundedness5.1}) show that
$\Sg_r=\{-N-1, -N,..., N-1, N\}$. Moreover, $G_{(n)}=G_{(-n-1)}$ by (\ref{boundedness3.6}).
Hence, (\ref{main cor1}) may be rewritten as
\begin{eqnarray}\label{main cor2}
F(z)&=&F_r(z)+2\pi i\sum_{n=-N-1}^NG_{(n)}(z) \nn\\
&=&F_r(z)+4\pi i\sum_{n=0}^NG_{(n)}(z)
\qquad (z\in W\setminus i\R)\,.
\end{eqnarray}
Notice that, since $\cz(r)>1$, if $iv=i(m+\frac{1}{2})$ with $m\in \mathbb N$, then $m\in \Sg_r$ for all $0<r<1$.

\begin{rem}
\label{rem:r}
By Lemma \ref{general formula}, for a fixed $0<r<1$, the function $F_r(z)$ is holomorphic
on the set of $z\in \C$ for which  (\ref{residue condition}) holds. The intersection of this
set with $i \R$ consists of the points $iv$ for which  $v \cz(r) \notin \mathbb Z+\frac{1}{2}$. The function $F_r(z)$ will be therefore
holomorphic in an open set containing $i (\mathbb Z+\frac{1}{2})$ provided $0<r<1$ is chosen so that
$ (\mathbb Z+\frac{1}{2}) \cap \cz(r)  (\mathbb Z+\frac{1}{2})\neq \emptyset$.
\end{rem}

The functions $G_{(n)}(z)$ are defined and even on $\C\setminus i\R$. In the following section we
fix $n\in \mathbb Z$ and determine a $2$-sheeted Riemann surface $\M_n$ covering $\C\setminus\{0\}$ and branched at $z=\pm i(n+\frac{1}{2})$. Then we prove that $G_{(n)}(z)$ extends as a meromorphic function to $\M_n$, with simple poles at the branching points above
$z=\pm i(n+\frac{1}{2})$.

\subsection{Meromorphic extension of the functions  $G_{(n)}(z)$}
\label{subsection:meroextGn}

The derivative of the map
\[
\C^2\ni (w,\zeta)\mapsto \zeta^2-w^2+1\in \C
\]
is equal to $(-2w,2\zeta)$. This is zero if and only if $w=\zeta=0$. Hence  the map
\[
(\C^2\setminus \{0,0\})\ni (w,\zeta)\mapsto \zeta^2-w^2+1\in \C
\]
is a submersion. Therefore the preimage of zero
\begin{equation}\label{preimageofzero}
\M=\{(w,\zeta);\ \zeta^2=w^2-1\}\subseteq \C^2
\end{equation}
is a complex submanifold. The fibers of the surjective holomorphic map
\begin{equation}\label{coveringMtoC}
\pi: \M\ni (w,\zeta)\mapsto w\in \C
\end{equation}
consist of two points $(w,\zeta)$ and $(w,-\zeta)$, if and only if $w\ne \pm 1$. If $w=\pm 1$, then the fibers consist of one point $(w,0)$.

Let $\Sg\subseteq \C$ and let $\wt \Sg\subseteq \M$ be the preimage of $\Sg$ in $\M$ under the map $\pi$ given by (\ref{coveringMtoC}). We say that a function $\wt f: \wt \Sg\to \C$ is a lift of a $f:\Sg\to \C$ if there is a holomorphic section $\sigma: \Sg \to \wt S$ of the restriction of
$\pi$ to $\wt \Sg$ so that
\[
\wt f(w,\zeta)=f(w) \qquad ((w,\zeta)\in \sigma(\Sg))\,.
\]
We see from Lemma \ref{scinverse} that
\begin{equation}\label{lift of sc}
(\sz\circ \cz^{-1})\,\wt{}:\M\ni (w,\zeta)\mapsto -\zeta\in \C
\end{equation}
is a lift of
\[
\sz\circ \cz^{-1}: \C\setminus [-1,1]\ni w\mapsto -\sqrt{w+1}\sqrt{w-1}\in \C
\]
and
\begin{equation}\label{lift of c}
(\cz^{-1})\,\wt{}:\M\ni (w,\zeta)\mapsto w-\zeta\in \C
\end{equation}
is a lift of
\[
\cz^{-1}: \C\setminus [-1,1]\ni w\mapsto w-\sqrt{w+1}\sqrt{w-1}\in \C\,.
\]
Both maps (\ref{lift of sc}) and (\ref{lift of c}) are holomorphic.

Let
\begin{eqnarray*}
&&\M_{(1,0)}=\{(w,\zeta);\ \zeta^2=w^2-1,\ \Re w>0\},\\
&&\M_{(-1,0)}=\{(w,\zeta);\ \zeta^2=w^2-1,\ \Re w<0\}.
\end{eqnarray*}
Then $\M_{(1,0)}$ is an open neighborhood of $(1,0)\in \M$ and $\M_{(-1,0)}$ is an open neighborhood of $(-1,0)\in \M$. Furthermore the following maps are local charts:
\begin{eqnarray*}
&&\M_{(1,0)}\ni (w,\zeta)\mapsto\zeta\in \C\setminus i\big((-\infty, -1] \cup [1,+\infty)\big),\ \ w=\sqrt{\zeta^2+1},\\
&&\M_{(-1,0)}\ni (w,\zeta)\mapsto\zeta\in \C\setminus  i\big((-\infty, -1] \cup [1,+\infty)\big),\ \ w=-\sqrt{\zeta^2+1}.
\end{eqnarray*}
\begin{lem}\label{lift of grn+}
Fix $n\in \Bbb Z$ and let
\begin{eqnarray*}
\M_{n}&=&\Big\{(z,\zeta)\in \C^\times \times(\C\setminus\{\pm i\});\ \zeta^2=\Big(\frac{i}{z}\big(n+\frac{1}{2}\big)\Big)^2-1\Big\}.
\end{eqnarray*}
The fibers of the map
\begin{eqnarray}\label{lift of grn+1}
\pi_n: \M_{n}\ni (z,\zeta)\mapsto z\in \C^\times
\end{eqnarray}
are $\{(z,\zeta), (z,-\zeta)\}$. In particular
\[
\M_{n}\setminus\big\{\pm \big(i\big(n+\tfrac{1}{2}\big),0)\Big\} \ni (z,\zeta)\mapsto z\in \C^\times \setminus \big\{\pm i\big(n+\tfrac{1}{2}\big)\big\}
\]
is a double cover.

Let $\M_{n,0}=\{(z,\zeta)\in \M_{n};\ z\notin i\R\}$.
The function $G_{(n)}$, (\ref{boundedness3}), lifts
to a holomorphic function $\wt G_{(n)}: \M_{n,0}\to \C$ and then extends to a meromorphic function
$\wt G_{(n)}: \M_{n}\to \C$ by the formula
\begin{eqnarray}\label{tboundedness3}
\wt G_{(n)}(z,\zeta)=
-3 \,\psi_{z}\big(\tfrac{i}{z}(n+\tfrac{1}{2})-\zeta\big)
\phi_{z,1}\big(\bxi\big(\tfrac{i}{z}(n+\tfrac{1}{2})-\zeta\big)\big)
\phi_{z,1}\big(\bxi^2 \big(\tfrac{i}{z}(n+\tfrac{1}{2})-\zeta\big)\big)\\
\cdot \frac{\tfrac{i}{z}(n+\tfrac{1}{2})}{-i\pi\zeta}\,. \nn
\end{eqnarray}
The function $\wt G_{(n)}$ is holomorphic on the set $\M_{n}\setminus\{(i(n+\frac{1}{2}),0), (-i(n+\frac{1}{2}),0)\}$.
It satisfies
\begin{equation}\label{eq:symmetrywtG}
\wt G_{(n)}(-z,-\zeta)=\wt G_{(n)}(z,\zeta)\qquad ((z,\zeta) \in \M_{n}).
\end{equation}

The following sets are open neighborhoods of $(i(n+\frac{1}{2}),0)$ and $(-i(n+\frac{1}{2}),0)$, respectively:
\begin{eqnarray*}
\Ug_{i(n+\frac{1}{2})}&=&\{(z,\zeta)\in \M_{n};\ (n+\tfrac{1}{2})\Im z >0\},\\
\Ug_{-i(n+\frac{1}{2})}&=&\{(z,\zeta)\in \M_{n};\  (n+\tfrac{1}{2})\Im z <0\}.
\end{eqnarray*}
Furthermore the following maps are local charts:
\begin{eqnarray}\label{charts}
\kappa_+:\Ug_{i(n+\frac{1}{2})}\ni (z,\zeta)&\mapsto& \zeta\in \C\setminus  i\big((-\infty, -1] \cup [1,+\infty)\big),\ \ z=\frac{i(n+\frac{1}{2})}{\sqrt{\zeta^2+1}},\\
\kappa_-:\Ug_{-i(n+\frac{1}{2})}\ni (z,\zeta)&\mapsto& \zeta\in \C\setminus  i\big((-\infty, -1] \cup [1,+\infty)\big),\ \ z=-\frac{i(n+\frac{1}{2})}{\sqrt{\zeta^2+1}}.\nn
\end{eqnarray}
The points $(\pm i(n+\frac{1}{2}),0)$ are simple poles of $\wt G_{(n)}$. The local expressions for $\wt G_{(n)}$ in terms of the charts (\ref{charts}) are
\begin{eqnarray}\label{expression in terms of charts 1}
\wt G_{(n)}\circ\kappa_+^{-1}(\zeta)&=&
3\, \psi_{\frac{i(n+\frac{1}{2})}{\sqrt{\zeta^2+1}}}(\sqrt{\zeta^2+1}-\zeta)
\\
&&\times \;
\phi_{{\frac{i(n+\frac{1}{2})}{\sqrt{\zeta^2+1}}},1}\big(\bxi(\sqrt{\zeta^2+1}-\zeta)\big)
\phi_{{\frac{i(n+\frac{1}{2})}{\sqrt{\zeta^2+1}}},1}\big(\bxi^2 (\sqrt{\zeta^2+1}-\zeta)\big)
\frac{\sqrt{\zeta^2+1}}{i\pi\zeta}
\nn
\end{eqnarray}
and
\begin{eqnarray}\label{expression in terms of charts 2}
\wt G_{(n)}\circ\kappa_-^{-1}(\zeta)&=&
-3 \psi_{\frac{i(n+\frac{1}{2})}{\sqrt{\zeta^2+1}}}(\sqrt{\zeta^2+1}+\zeta)\\
&&\times \; \phi_{{\frac{i(n+\frac{1}{2})}{\sqrt{\zeta^2+1}}},1}\big(\bxi(\sqrt{\zeta^2+1}+\zeta)\big)
\phi_{{\frac{i(n+\frac{1}{2})}{\sqrt{\zeta^2+1}}},1}\big(\bxi^2 (\sqrt{\zeta^2+1}+\zeta)\big)\;
\frac{\sqrt{\zeta^2+1}}{i\pi\zeta}.\nn
\end{eqnarray}
Furthermore,
\begin{eqnarray}
\label{eq:residue-n-plus}
2\pi i \, \Res_{\zeta=0}\wt G_{(n)}\circ\kappa_+^{-1}(\zeta)&=&\frac{3}{2\pi} \Big(n+\frac{1}{2}\Big)^2
\psi_{i(n+\tfrac{1}{2})}(1)
\end{eqnarray}
and
\begin{eqnarray}
\label{eq:residue-n-minus}
2\pi i  \, \Res_{\zeta=0}\wt G_{(n)}\circ\kappa_-^{-1}(\zeta)&=&
-\frac{3}{2\pi} \Big(n+\frac{1}{2}\Big)^2
\psi_{i(n+\tfrac{1}{2})}(1)\,.
\end{eqnarray}
\end{lem}
\begin{prf}
We replace $w$ by $\frac{i}{z}(n+\frac{1}{2})$ in  the discussion preceding the lemma. This explains the structure of the covering (\ref{lift of grn+1}). Formula (\ref{tboundedness3}) follows from that discussion and (\ref{boundedness3}). The function (\ref{tboundedness3}) might be singular at some point $(z,\zeta)$ with $\zeta\ne 0$ if and only if
\begin{equation}\label{requirement1.1mod}
\phi_{z,1}\big(\bxi\big(\tfrac{i}{z}(n+\tfrac{1}{2})-\zeta\big)\big)
\phi_{z,1}\big(\bxi^2 \big(\tfrac{i}{z}(n+\tfrac{1}{2})-\zeta\big)\big)
\end{equation}
is singular at $(z,\zeta)$, $\zeta\ne 0$. Fix $k \in \{1,2\}$ and suppose one term
\begin{equation}\label{requirement1.1.1mod}
\phi_{z,1}\big(\bxi^k \big(\tfrac{i}{z}(n+\tfrac{1}{2})-\zeta\big)\big)
\end{equation}
is singular at $(z,\zeta)$, $\zeta\ne 0$.
Then
\[
z\cz(\bxi^k(\tfrac{i}{z}(n+\tfrac{1}{2})-\zeta))\in i(\Bbb Z+\tfrac{1}{2}),
\]
which means
\begin{equation}\label{requirement1.5mod}
\cz(\bxi^k(\tfrac{i}{z}(n+\tfrac{1}{2})-\zeta))\in \tfrac{i}{z}(\Bbb Z+\tfrac{1}{2}).
\end{equation}
Assume first that
\begin{equation}\label{requirement2}
z\notin i(-\infty, -\left|n+\tfrac{1}{2}\right|]\cup i[\left|n+\tfrac{1}{2}\right|, \infty)
\end{equation}
or equivalently
\begin{equation}\label{requirement3}
\tfrac{i}{z}\left(n+\tfrac{1}{2}\right)\notin [-1,1].
\end{equation}
Then $\cz^{-1}\left(\frac{i}{z}\left(n+\frac{1}{2}\right)\right)$ is in the open unit disc $\Dg_1$ with the zero removed. It follows that
\begin{equation}
\label{eq:inD1}
\bxi^k\cz^{-1}\left(\tfrac{i}{z}\left(n+\tfrac{1}{2}\right)\right)\in \Dg_1\setminus \{0\}=\cz^{-1}(\C\setminus [-1,1])\,.
\end{equation}
Because of (\ref{requirement3}), Lemma \ref{scinverse} implies that
\begin{equation}\label{requirement4}
\tfrac{i}{z}(n+\tfrac{1}{2})-\zeta=\cz^{-1}(\tfrac{i}{z}(n+\tfrac{1}{2}))\ \ \text{or}\ \ \tfrac{i}{z}(n+\tfrac{1}{2})-\zeta=(\cz^{-1}(\tfrac{i}{z}(n+\tfrac{1}{2})))^{-1}.
\end{equation}
In the first case, (\ref{requirement1.5mod}) and (\ref{eq:inD1}) yield
\[
\bxi^k \cz^{-1}\left(\tfrac{i}{z}\left(n+\tfrac{1}{2}\right)\right)\in \cz^{-1}\left( \tfrac{i}{z}\left(\Bbb Z+\tfrac{1}{2}\right)\setminus [-1,1]\right).
\]
Since $\bxi^k\ne \pm 1$, Lemma \ref{rays and circles} shows that this is impossible.  In the second case,
\begin{eqnarray*}
\cz(\bxi^k(\tfrac{i}{z}(n+\tfrac{1}{2})-\zeta))&=&\cz(\bxi^{k}(\cz^{-1}(\tfrac{i}{z}(n+\tfrac{1}{2})))^{-1})
=\cz((\bxi^{-k}\cz^{-1}(\tfrac{i}{z}(n+\tfrac{1}{2})))^{-1})\\
&=&\cz(\bxi^{-k}\cz^{-1}(\tfrac{i}{z}(n+\tfrac{1}{2}))).
\end{eqnarray*}
Hence,  (\ref{requirement1.5mod}) and (\ref{eq:inD1}) yield
\[
\bxi^{-k}\cz^{-1}(\tfrac{i}{z}(n+\tfrac{1}{2}))\in \cz^{-1}\left( \tfrac{i}{z}(\Bbb Z+\tfrac{1}{2})\setminus [-1,1]\right).
\]
Since $\bxi^{-k}\ne \pm 1$, Lemma \ref{rays and circles} shows that this is impossible.

Assume from now on that
\begin{equation}\label{requirement5}
z\in i(-\infty, -\left|n+\tfrac{1}{2}\right|]\cup i[\left|n+\tfrac{1}{2}\right|, \infty)
\end{equation}
or equivalently that
\begin{equation}\label{requirement5.5}
\tfrac{i}{z}\left(n+\tfrac{1}{2}\right)\in [-1,1].
\end{equation}
Let us find all the points  $(z,\zeta)$, $\zeta\ne 0$, where the term
\begin{equation}\label{requirement1.1.1.1}
\phi_{z,1}\big(\bxi^k(\tfrac{i}{z}(n+\tfrac{1}{2})-\zeta)\big)
\end{equation}
is singular. These are the points where (\ref{requirement1.5mod}) holds.

For an integer $m$ let $x_m=\frac{i}{z}\left(m+\frac{1}{2}\right)$. Then there is a unique $\epsilon=\pm 1$ such that
\[
\tfrac{i}{z}(n+\tfrac{1}{2})=x_n\ \ \text{and}\ \ \zeta=-\epsilon i\sqrt{1-x_n^2}.
\]
Hence  (\ref{requirement1.5mod}) is equivalent to the statement that there is $m\in\Bbb Z$  such that
\begin{equation}\label{requirement5.6mod}
\cz(\bxi^k(x_n+\epsilon i\sqrt{1-x_n^2}))=x_m.
\end{equation}
By (\ref{requirement5.5}), the left-hand side of this equation is equal to cosine of some real angle. Hence,
\begin{equation}\label{requirement5.6.1}
x_m\in [-1,1].
\end{equation}
Notice that $\bxi^k=\delta\frac{1}{2}+ i\frac{\sqrt{3}}{2}$, where $\delta=1$ if $k=1$ and  $\delta=-1$ if $k=2$.
Therefore (\ref{requirement5.6mod}) is equivalent to
\begin{eqnarray}\label{requirement5.8}
&&\delta\tfrac{1}{2}x_n-\epsilon\tfrac{\sqrt{3}}{2}\sqrt{1-x_n^2}=x_m\,,
\end{eqnarray}
which implies
\begin{eqnarray}\label{requirement5.8.1}
x_n^2 -\delta x_nx_m+x_m^2=\frac{3}{4}.
\end{eqnarray}
Let $x=\tfrac{i}{z}$. Then (\ref{requirement5.8.1}) is equivalent to
\[
x^2=\frac{3/4}{\left(n+\frac{1}{2}\right)^2 -\delta\left(n+\frac{1}{2}\right)\left(m+\frac{1}{2}\right)+\left(m+\frac{1}{2}\right)^2}
\]
Hence,
\[
1-x_n^2=\frac{\left(\left(m+\frac{1}{2}\right)-\delta\frac{1}{2}\left(n+\frac{1}{2}\right)\right)^2}{\left(n+\frac{1}{2}\right)^2 -\delta\left(n+\frac{1}{2}\right)\left(m+\frac{1}{2}\right)+\left(m+\frac{1}{2}\right)^2}.
\]
The numerator of this fraction cannot be zero. Hence $x_n^2<1$. Thus (\ref{requirement5.5}) is actually equivalent to
\begin{equation}\label{requirement5.5.1}
\tfrac{i}{z}\left(n+\tfrac{1}{2}\right)\in (-1,1).
\end{equation}
Write (\ref{requirement5.8}) as
\[
x\left(\delta\tfrac{1}{2}\left(n+\tfrac{1}{2}\right)-\left(m+\tfrac{1}{2}\right)\right)=
\epsilon\tfrac{\sqrt{3}}{2}\sqrt{1-x_n^2}.
\]
This is equivalent to
\begin{eqnarray}\label{requirement5.7.2}
z&=&-i\frac{2}{\sqrt{3}} \left(2\left(m+\frac{1}{2}\right)-\delta\left(n+\frac{1}{2}\right)\right)
\frac{\sqrt{\left(n+\frac{1}{2}\right)^2 -\delta\left(n+\frac{1}{2}\right)\left(m+\frac{1}{2}\right)+\left(m+\frac{1}{2}\right)^2}}
{\epsilon\left|2\left(m+\frac{1}{2}\right)-\delta\left(n+\frac{1}{2}\right)\right|}\nn\\
\zeta&=&\epsilon i \frac{1}{2}\;\frac{\left|2\left(m+\frac{1}{2}\right)-\delta\left(n+\frac{1}{2}\right)\right|}
{\sqrt{\left(n+\frac{1}{2}\right)^2 -\delta\left(n+\frac{1}{2}\right)\left(m+\frac{1}{2}\right)+\left(m+\frac{1}{2}\right)^2}}.
\end{eqnarray}
It is easy to check that the $z$ given by (\ref{requirement5.7.2}) satisfies the following inequalities:
\begin{eqnarray}\label{requirement5.7.2.1}
|z|>\left|n+\tfrac{1}{2}\right|\ \ \text{and}\ \ |z|>\left|m+\tfrac{1}{2}\right|.
\end{eqnarray}
Hence, (\ref{requirement5.6.1}) actually reads
$
x_m\in (-1,1)
$.
Thus given any integer $m$ there are two points $(z,\zeta), (-z,-\zeta)\in \M_n$ (corresponding to $\epsilon=\pm 1$) such that
(\ref{requirement5.6mod}) holds. This completes our task of finding all the possible singularities of the function (\ref{requirement1.1.1.1}).

We now prove that all these possible singularities are in fact removable.

Suppose $(z,\zeta)$ satisfies (\ref{requirement5.7.2}).
Set
\begin{equation}
\label{eq:w}
w=\tfrac{i}{z} (n+\tfrac{1}{2})-\zeta\,.
\end{equation}
According to (\ref{requirement5.7.2}), there is an integer $m$
so that
\begin{eqnarray}
\label{eq:zzeta}
z \zeta= -\tfrac{1}{\sqrt{3}} \big(2(m+\tfrac{1}{2})-\delta (n+\tfrac{1}{2})\big)\,.
\end{eqnarray}
where $\delta=1$ if $\bxi w$ is a pole of $\phi_{z,1}$ and $\delta=-1$ if $\bxi^2 w$ is a pole of $\phi_{z,1}$.
Choose $k \in \{1,2\}$ so that $\bxi^k w$ is a pole of $\phi_{z,1}$.
Then $\bxi^k=\delta\frac{1}{2}+i\frac{\sqrt{3}}{2}$ and $\bxi^{3-k}=-\delta\frac{1}{2}+i\frac{\sqrt{3}}{2}$.
Hence, by (\ref{requirement5.6mod}),
\begin{equation}
\label{eq:m-for pole}
z \cz(\bxi^k w)=i(m+\tfrac{1}{2})\,.
\end{equation}
Notice that
\begin{eqnarray*}
zw&=&i(n+\tfrac{1}{2})-z\zeta \\
&=& i(n+\tfrac{1}{2}) +\tfrac{1}{\sqrt{3}} (2(m+\tfrac{1}{2})-\delta (n+\tfrac{1}{2})) \\
&=& \tfrac{2}{\sqrt{3}} \left[ \big(i\tfrac{\sqrt{3}}{2}-\tfrac{\delta}{2}\big)(n+\tfrac{1}{2})+(m+\tfrac{1}{2}) \right] \\
&=& \tfrac{2}{\sqrt{3}} (\bxi^{3-k} (n+\tfrac{1}{2})+(m+\tfrac{1}{2}))\,.
\end{eqnarray*}
Using the relation $\sz(\alpha)+\cz(\alpha)=\alpha$, we compute
\begin{eqnarray}
\label{eq:ws-pole}
z \sz(\bxi^k w)&=& z \bxi^k w-z \cz(\bxi^k w) \nn\\
&=& \tfrac{2}{\sqrt{3}} \bxi^k (\bxi^{3-k} (n+\tfrac{1}{2})+(m+\tfrac{1}{2})) -i(m+\tfrac{1}{2}) \nn\\
&=&\tfrac{2}{\sqrt{3}} \big(\tfrac{1}{2}\delta (m+\tfrac{1}{2})-(n+\tfrac{1}{2})\big),
\end{eqnarray}
because $\bxi^3=-1$ and
$\tfrac{2}{\sqrt{3}}  \bxi^k -i= \tfrac{2}{\sqrt{3}} \big(\tfrac{1}{2}\delta+\tfrac{\sqrt{3}}{2}\,i\big)-i =\frac{1}{\sqrt{3}} \, \delta\,.$

We now prove that $z \cz(\bxi^{3-k}w) \in i\mathbb{Z}$.
Recall that for all $\alpha, \beta\in \C^\times$ we have
\begin{eqnarray*}
\cz(\alpha\beta)&=&\cz(\alpha)\cz(\beta)+\sz(\alpha)\sz(\beta)\,,\\
\sz(\alpha\beta)&=&\sz(\alpha)\cz(\beta)+\cz(\alpha)\sz(\beta)\,.
\end{eqnarray*}
Notice also that
$$\cz(\bxi^{3-2k})=\tfrac{1}{2} \quad \text{and}
\quad \sz(\bxi^{3-2k})=i\delta \tfrac{\sqrt{3}}{2}\,.$$
Using (\ref{eq:m-for pole})
and (\ref{eq:ws-pole}), we obtain
\begin{eqnarray*}
z \cz(\bxi^{3-k} w)&=& z \cz(\bxi^{3-2k} \bxi^k w)\\
&=& z \cz(\bxi^{3-2k}) \cz(\bxi^k w)+z \sz(\bxi^{3-2k}) \sz(\bxi^k w)\\
&=& i\tfrac{1}{2} (m+\tfrac{1}{2})+i \delta \big(\tfrac{1}{2}
\delta (m+\tfrac{1}{2})-(n+\tfrac{1}{2})\big)\\
&=& i  (m+\tfrac{1}{2})-i \delta  (n+\tfrac{1}{2}) \in i\mathbb{Z}\,.
\end{eqnarray*}
Thus $\phi_{z,1}(\bxi^{3-k} w)=\phi_{z,1}\big(\bxi^{3-k} (\frac{i}{z}(n+\frac{1}{2})-\zeta)\big)=0$
and the function $\wt G_{(n)}$ extends to be holomorphic near
$(z,\zeta)\in \M_n$.

The symmetry property (\ref{eq:symmetrywtG}) as well as the local expressions of $\wt G_{(n)}$ in terms of the charts are immediate from (\ref{tboundedness3})
and the parity properties of $\psi_z$ and $\phi_{z,1}$ in (\ref{even phiwu}) and  (\ref{even phiwu bis}).

The first residue at $\zeta=0$ is equal to
\begin{eqnarray*}
\Res_{\zeta=0}\wt G_{(n)}\circ\kappa_+^{-1}(\zeta)&=&
3\psi_{i(n+\frac{1}{2})}(1)\phi_{i(n+\frac{1}{2}),1}(\bxi) \phi_{i(n+\frac{1}{2}),1}(\bxi^2)
\frac{1}{i\pi}\\
&=& -\frac{3}{4i\pi} \left(n+\frac{1}{2}\right)^2 \psi_{i(n+\frac{1}{2})}(1) \,\th^2\Big(\frac{i\pi}{2} \big(n+\frac{1}{2}\big) \Big),
\end{eqnarray*}
which proves the result as
$\th\big(\frac{\pi i}{2}(n+\tfrac{1}{2}) \big)=(-1)^n i$.
One computes the second residue similarly, using (\ref{even phiwu}).
\end{prf}

Recall that the residues may depend on the choice of the chart. However the type of the singularity does not.

\begin{rem}
\label{rem:principal branch wtG}
Let $n\in \mathbb N$ be fixed. By Lemma  \ref{two square roots},
\begin{equation}\label{eq:sigman+}
\sigma_n^+: \C \setminus i\big((-\infty, -(n+\tfrac{1}{2})]\cup [n+\tfrac{1}{2}, +\infty)\big) \ni z \mapsto (z,\zeta_n^+(z)) \in \M_n\,,
\end{equation}
where
\begin{equation}\label{eq:zetan+}
\zeta_n^+(z)=\sqrt{\tfrac{i}{z}(n+\tfrac{1}{2})+1} \sqrt{\tfrac{i}{z}(n+\tfrac{1}{2})-1}\,,
\end{equation}
is a holomorphic section of the projection map (\ref{lift of grn+1}) so that
\begin{equation}\label{eq:wtG zetan+}
\wt G_{(n)} \circ \sigma_n^+ =G_{(n)}\,.
\end{equation}
This in particular implies that $G_{(n)}$ is holomorphic on $\C \setminus i\big((-\infty, -(n+\tfrac{1}{2})]\cup [n+\tfrac{1}{2}, +\infty)\big)$ and that (\ref{eq:wtG zetan+}) extends to this domain by analyticity.
The image of $\sigma_n^+$ is usually refereed to as the physical sheet  (or principal sheet) of $\M_n$.  The image in $\M_n$ of the map
\begin{equation*}
\sigma_n^-: \C \setminus i\Big((-\infty, -(n+\tfrac{1}{2})]\cup [(n+\tfrac{1}{2}), +\infty)\Big) \ni z \mapsto (z,\zeta^-_n(z)) \in \M_n
\end{equation*}
is the nonphysical sheet. In the above equation, we have set
\begin{equation}\label{eq:zetanminus}
\zeta_n^-(z)= -\zeta_n^+(z)\,.
\end{equation}
For $v \in \R^+$ let $\zeta_n^\pm(-iv)=\zeta_n^\pm(-iv+0)$. Then, by Lemma \ref{scinverse},
\begin{equation}\label{eq:zetan+-iv}
\zeta_n^+(-iv)=
\begin{cases}
\sqrt{\big(\tfrac{n+1/2}{v}\big)^2-1} &\text{if $0<v\leq n+\tfrac{1}{2}$\,,}\\
 i\sqrt{1-\big(\tfrac{n+1/2}{v}\big)^2} &\text{if $v> n+\tfrac{1}{2}$\,,}\\
\end{cases}
\end{equation}
is the physical lift of $-i\R^+$ in $\M_n$. The lift of $-i\R^+$ in $\M_n$ is the branched curve $\R^+\ni v \mapsto \zeta_n^\pm(-iv)$. It is represented in Figure 1.

\begin{center}
\begin{figure}
\vskip -3.5truecm

\includegraphics[width=16cm]{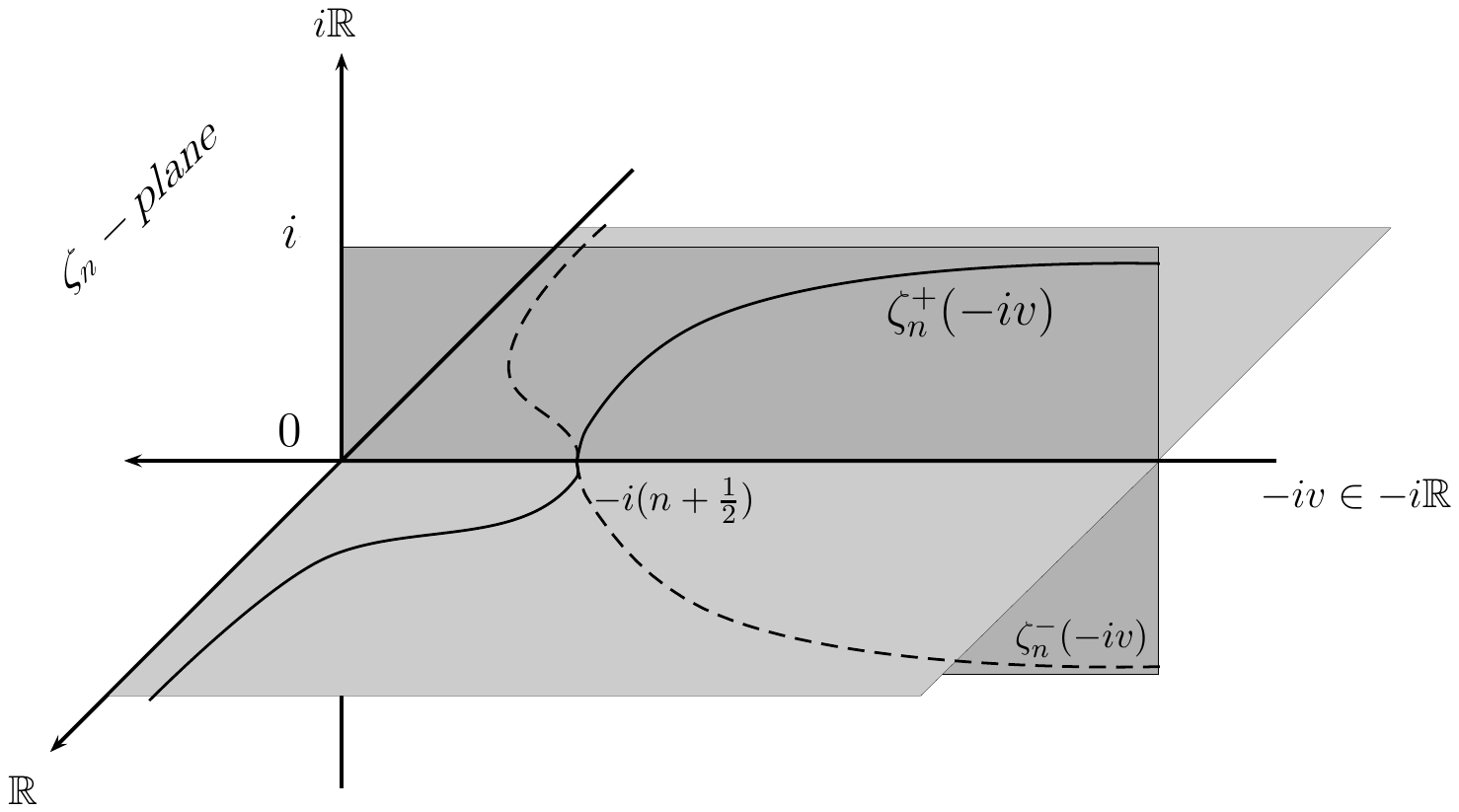}\qquad \null

\vskip -13.5truecm

\caption{The branched curve $v \in \R^+ \to  \pm\zeta_n^+(-iv)$}
\end{figure}
\end{center}

\end{rem}


\subsection{Meromorphic extension of $F$}
\label{subsection:mero-ext-F}

According to (\ref{main cor2}), locally, in a neighborhood $W$ of each point  $-iv \in -i\R^+$, the function $F$ can be written as
\begin{equation}
\label{eq:Fresidues1}
F(z)=F_r(z)+4\pi i\sum_{n=0}^N G_{(n)}(z)
\qquad (z\in W\setminus i\R)\,,
\end{equation}
where $0<r<1$ and $N$ depend $v$, and $F_r$ is holomorphic in $W$.

In this subsection we determine a meromorphic extension of $F$ above $-i\R^+$ by ``putting together" the meromorphic extensions of the functions $G_{(n)}(z)$ determined in Lemma \ref{lift of grn+}.
To do this, we need more precise information on the parameters $r$ and $N$ occurring in
(\ref{eq:Fresidues1}).
\begin{lem}
\label{lemma:upper-limit-sum}
Fix $N\in \mathbb N$. Suppose $0<r<1$ is chosen so that $\cz(r)<1+\frac{1}{2N+3}$.
Then for every $0<v<  N+\tfrac{3}{2}$ there is an open neighborhood $W_v$ of $-iv$ so that
$$\Sg_{r,z}\cap \mathbb N=\{n\in \mathbb N; -i(n+\tfrac{1}{2})\in W_v \Eg_{\cz(r),\sz(r)}\}=\{0,1,\dots,N_v\}  \qquad (z\in W_v\setminus i\R)\,,$$
where $N_v\leq \floor{v}$. Here $\floor{v}$ denotes the largest integer less than or equal to $v$.

Furthermore, if $\floor{v}+\frac{1}{2}\leq v$ then $N_v= \floor{v}$.
If $v<\floor{v}+\frac{1}{2}$ and we choose $0<r=r_v<1$ so that it satisfies the additional condition $\cz(r_v)<\frac{\floor{v}+\frac{1}{2}}{v}$, then $N_v= \floor{v}-1$.
\end{lem}
\begin{prf}
The ellipse $\Eg_{\cz(r),\sz(r)}$ being invariant under sign change, we can work with
$iv$ instead of $-iv$.

Let $iu\in i\R^+$. Then $iu \in iv \Eg_{\cz(r),\sz(r)}$ if and only if $u<\cz(r) v$. This shows that
$i[u,+\infty) \cap iv \Eg_{\cz(r),\sz(r)}=\emptyset$ if $iu \notin iv\Eg_{\cz(r),\sz(r)}$. Moreover,
since $\cz(r)>1$, we have that $i[0,v] \subseteq iv \Eg_{\cz(r),\sz(r)}$.
In particular,
\begin{equation}\label{eq:integerv-in-Sr}
i\big(\floor{v}+\tfrac{1}{2}\big)\in iv \Eg_{\cz(r),\sz(r)}
\end{equation}
 if $\floor{v}+\frac{1}{2}\leq v$.
Notice also that  $iu \notin iv \overline{\Eg_{\cz(r),\sz(r)}}$ if and only if $u>\cz(r) v$.

Suppose $\cz(r)<1+\frac{1}{2N+3}$ and let $0<v<  N+\tfrac{3}{2}$. Then
$1+\frac{1}{2v}\geq 1+\frac{1}{2N+3} > \cz(r)$.
So $(\floor{v}+1)+\frac{1}{2}>v+\frac{1}{2}> v\cz(r)$, i.e.
\begin{equation}\label{eq:integerv-notin-Sr}
 i\big((\floor{v}+1)+\tfrac{1}{2})\notin iv \overline{\Eg_{\cz(r),\sz(r)}}\,.
\end{equation}
The relations (\ref{eq:integerv-in-Sr}) and (\ref{eq:integerv-notin-Sr}) still hold when we replace $iv \Eg_{\cz(r),\sz(r)}$
with $z \Eg_{\cz(r),\sz(r)}$ with $z$ in a sufficiently small neighborhood $W$ of $iv$. We also take $W$ small enough so that
$\Sg_{r,z}$ is independent of $z \in W\setminus i\R$; see Corollary \ref{boundedness} and Lemma \ref{connected nbh}.
The extension of (\ref{eq:integerv-notin-Sr}) obtained in this way shows that $N_v\leq  \floor{v}$.
Thus $N_v= \floor{v}$ if $\floor{v}+\frac{1}{2}\leq v$.

If $v<\floor{v}+\frac{1}{2}$, then we can choose $0<r=r_v<1$ so that it also satisfies the condition $\cz(r_v)<\frac{\floor{v}+\frac{1}{2}}{v}$. So $v\cz(r_v)<\floor{v}+\frac{1}{2}$ yields $N_v<\floor{v}$. Thus
$N_v=\floor{v}-1$ in this case.
\end{prf}

\begin{cor}
\label{cor:FextendedWv}
Let $N\in \mathbb N$ and $m\in \{0,1,\dots,N\}$. Let $v \in \big[m+\frac{1}{2}, m+\frac{3}{2}\big)$. Then there is an open neighborhood $W_v$ centered at $-iv$ and $0<r_v<1$
so that
 \begin{equation}
\label{eq:Fresidues22}
F(z)=F_{r_v}(z)+4\pi i\sum_{n=0}^{m} G_{(n)}(z)
\qquad (z\in W_v\setminus i\R)\,,
\end{equation}
where $F_{r_v}$ is holomorphic in $W_v$ and $G_{(n)}$ is the function defined in (\ref{boundedness3}).
\end{cor}
\begin{prf}
This is immediate from Lemma \ref{lemma:upper-limit-sum}. Indeed, for
$v \in  \big[m+\frac{1}{2}, m+1\big)=\big[\floor{v}+\frac{1}{2}, \floor{v}+1\big)$ we have $N_v=\floor{v}=m$,
whereas for $v \in \big[m+1, m+\frac{3}{2}\big)=\big[\floor{v}, \floor{v}+\frac{1}{2}\big)$ we have $N_v=\floor{v}-1=(m+1)-1=m$.
\end{prf}

In the following we suppose that we have fixed the $W_v$ as in Corollary
\ref{cor:FextendedWv}. Moreover, by possibly further shrinking them, we may also assume that
every $W_v$ is an open disk centered at $-iv$ and that
\begin{eqnarray}
W_v\cap i\R  \subseteq &-i\big(m+\tfrac{1}{2}, m+\tfrac{3}{2}\big)\qquad &\text{if $v\in \big(m+\tfrac{1}{2}, m+\tfrac{3}{2}\big)$\,,}
\label{eq:assWvint}\\
W_v\cap i\R  \subseteq &-i(m, m+1) \hfill \qquad &\text{if $v=m+\tfrac{1}{2}$}\,.
\label{eq:assWvbd}
\end{eqnarray}

\begin{cor}
\label{cor:forFm}
Let $m\in \{0,1,\dots,N\}$ and $v \in \big[m+\frac{1}{2}, m+\frac{3}{2}\big)$.
Suppose that $W_v\cap W_{v'}\neq \emptyset$ for some $v'\in \big[\frac{1}{2}, N+\frac{3}{2}\big)$.

If $v \in \big(m+\frac{1}{2}, m+\frac{3}{2}\big)$, then $v' \in \big[m+\frac{1}{2}, m+\frac{3}{2}\big)$. Moreover,
$$F_{r_{v'}}(z)=F_{r_v}(z) \qquad (z \in W_v\cap W_{v'}).$$

If $v=m+\frac{1}{2}$,  then $v' \in \big(m-\frac{1}{2}, m+\frac{3}{2}\big)$.
Moreover, if $v' \in \big[m+\frac{1}{2}, m+\frac{3}{2}\big)$, then
$$F_{r_{v'}}(z)=F_{r_v}(z) \qquad (z \in W_v\cap W_{v'});$$
if $v' \in \big(m-\frac{1}{2}, m+\frac{1}{2}\big)$, then
$$F_{r_{v'}}(z)=F_{r_v}(z)+4\pi i \, G_{(m)}(z) \qquad (z \in W_v\cap W_{v'}).$$
\end{cor}
\begin{prf}
Observe first that $W_v\cap W_{v'}$ is connected. If $v \in \big(m+\frac{1}{2}, m+\frac{3}{2}\big)$ and $W_v\cap W_{v'}\neq \emptyset$, then $v'\in \big[m+\frac{1}{2}, m+\frac{3}{2}\big)$
by (\ref{eq:assWvint}). Moreover,  (\ref{eq:Fresidues22}) gives $F_{r_{v'}}(z)=F_{r_v}(z)$ for
$z \in (W_v\cap W_{v'}) \setminus i\R$. This equality extends to $W_v\cap W_{v'}$ by analyticity.

Suppose now $v=m+\frac{1}{2}$. The assumption (\ref{eq:assWvbd}) ensures that if
$W_v\cap W_{v'}\neq \emptyset$ then $v' \in \big(m-\frac{1}{2}, m+\frac{3}{2}\big)$. When
$v' \in \big[m+\frac{1}{2}, m+\frac{3}{2}\big)$, the equality of the functions $F_r$ follows as
above. If $v' \in \big[m-\frac{1}{2}, m+\frac{1}{2}\big)$, then
$$F(z)=F_{r_{v'}}(z)+4\pi i\sum_{n=0}^{m-1} G_{(n)}(z)
\qquad (z\in W_{v'}\setminus i\R)\,.$$
Therefore $F_{r_{v'}}(z)=F_{r_v}(z)+4\pi i \; G_{(m)}(z)$ for $z \in (W_v\cap W_{v'})\setminus i\R$.
By (\ref{eq:assWvint}), we have $W_{v'} \cap i\R \subset -i \big(m-\frac{1}{2}, m+\frac{1}{2}\big)$.
So $G_{(m)}$ is holomorphic on $W_v\cap W_{v'}$ (see Remark \ref{rem:principal branch wtG}).
Since $W_v\cap W_{v'}$ is connected, the previous equality holds for all $z \in
W_v\cap W_{v'}$.
\end{prf}

For every integer $0\leq m \leq N$ we define
\begin{eqnarray}
\label{eq:Wm}
W_{(m)}&=&\bigcup_{v\in [m+\frac{1}{2}, m+\frac{3}{2})} W_v\\
\label{eq:Fm}
 F_{(m)}(z)&=&F_{r_v}(z)  \qquad (v\in \big[m+\tfrac{1}{2}, m+\tfrac{3}{2}\big), \,z\in W_v)\,.
\end{eqnarray}
Then $W_{(m)}$ is an open neighborhood of $-i\big[m+\frac{1}{2}, m+\frac{3}{2}\big)$
so that $W_{(m)}\cap i\R\subseteq -i\big(m, m+\frac{3}{2}\big)$.
Moreover, setting
\begin{eqnarray}
\label{eq:W'm}
W'_{(m)}&=&\bigcup_{v\in (m+\frac{1}{2}, m+\frac{3}{2})} W_v\,,
\end{eqnarray}
we have $W'_{(m)}\cap i\R =
-i\big(m+\frac{1}{2}, m+\frac{3}{2}\big)$.
By Corollary \ref{cor:forFm},
$F_{(m)}$ is a holomorphic extension of $F_{r_v}$ to $W_{(m)}$ for every $v \in \big[m+\frac{1}{2}, m+\frac{3}{2}\big)$.

It is convenient to introduce similar notation for a neighborhood of $(0,\frac{1}{2})$. We therefore choose for every $v\in (0,\frac{1}{2})$ and open ball $W_v$ centered at $-iv$ so that $W_v\cap i\R\subseteq (0,\frac{1}{2})$ and define:
\begin{eqnarray}
\label{eq:Wminusone}
W_{(-1)}&=&W'_{(-1)}=\bigcup_{v\in (0,\frac{1}{2})} W_v\\
\label{eq:Fminusone}
 F_{(-1)}(z)&=&F(z)  \qquad (z\in W_{(-1)})\,.
\end{eqnarray}
Observe that $W_{(-1)} \cap i\R=-i(0,\frac{1}{2})$. Moreover, for every $m\in \{-1,0,\dots,N-1\}$ the intersection $W_{(m)}\cap W_{(m+1)}=W_{(m)}\cap W_{m+\frac{3}{2}}$ is a nonempty open connected set.
\begin{cor}
\label{cor:FextendedMainSection}
For every integer $m\geq -1$ we have
 \begin{equation}
\label{eq:Fresidues2}
F(z)=F_{(m)}(z)+4\pi i\sum_{n=0}^{m} G_{(n)}(z)
\qquad (z\in W_{(m)}\setminus i\R)\,,
\end{equation}
where $F_{(m)}$ is holomorphic in $W_{(m)}$, the $G_{(n)}$ are as in (\ref{boundedness3}), and empty sums are defined to be equal to $0$.
Consequently, for every integer $m\geq -1$,
\begin{equation}
\label{eq:FmFm-1}
F_{(m)}(z)=F_{(m+1)}(z)+4\pi i \;G_{(m+1)}(z) \qquad (z \in W_{(m)}\cap W_{(m+1)}=W_{(m)}\cap W_{m+\frac{3}{2}})\,.
\end{equation}
\end{cor}
\begin{prf}
This is an immediate consequence of Corollary \ref{cor:FextendedWv} and (\ref{eq:Wm})--(\ref{eq:Fminusone}).
\end{prf}

The meromorphic continuation of $F$ across $-i[\frac{1}{2},+\infty)$ will be done per stages, on Riemann surfaces constructed from the
$\M_n$'s. Recall that $\M_n$ denotes the Riemann surface to which $G_{(n)}$ lifts as a meromorphic function.
For an integer $N\geq 0$, let
\begin{equation}
\label{eq:MN}
\M_{(N)}
=\left\{(z,\zeta)\in \C^\times \times \C^{N+1};\ \zeta=(\zeta_0, \zeta_1,\dots, \zeta_N);\ (z,\zeta_n)\in \M_n,\ 0\leq n\leq N\right\}.
\end{equation}
Then $\M_{(N)}$ is a Riemann surface, and the map
\begin{eqnarray}\label{finite cover 1}
&&\pi_{(N)}: \M_{(N)} \ni(z,\zeta)\mapsto z\in \C^\times
\end{eqnarray}
is a holomorphic $2^{N+1}$-to-$1$ cover, except when $z\in\{\pm i(n+\frac{1}{2});\ n\in \mathbb N\,, 0\leq n\leq N\}$. (This may be seen by checking, as in (\ref{preimageofzero}), that if $a_1, a_2, \dots, a_k$ are non-zero complex numbers whose squares are mutually distinct, then
\begin{eqnarray}\label{the preimage of zero again}
\{(w,\zeta_1, \zeta_2, \dots, \zeta_k);\ \zeta_j^2=a_j^2 w^2-1,\ 1\leq j\leq k\}\subseteq \C^{k+1}
\end{eqnarray}
is a one-dimensional complex submanifold with all the required properties.).

The $\zeta$-coordinates $\zeta_n$ ($0\leq n \leq N$) of the points of the fiber of $-iv \in -i\R^+$ in $\M_{(N)}$ are uniquely determined
by the condition that $(-iv,\zeta_n) \in \M_n$.
This means that
\begin{eqnarray} \label{eq:zetam}
\zeta^2_n=\left(\frac{n+\frac{1}{2}}{v}\right)^2-1 \qquad (0\leq n\leq N)\,.
\end{eqnarray}
If $v \notin \{m+\frac{1}{2}; m\in \mathbb N\,, 0\leq m\leq N\}$, we get exactly $2^{N+1}$ points, corresponding to the two sign choices for each $\zeta_n$.
Hence the fiber of $-iv \in -i\R^+ \setminus  \{-i\big(m+\frac{1}{2}\big); m\in \mathbb N\,, 0\leq m\leq N\}$ consists of the points
\begin{equation}
\label{eq:fiber -iv generic}
\big( -iv, \zeta^\pm_0(-iv), \dots, \zeta^\pm_N(-iv)\big)
\end{equation}
where $\zeta^\pm_n(-iv)$ is as in (\ref{eq:zetan+-iv}).
If $v=m+\frac{1}{2}$ for some $m\in\{0,\dots,N\}$, then the $\zeta_m$-coordinate of the points of
the fiber of $-i(m+\frac{1}{2})$ is zero, whereas (\ref{eq:zetam}) has precisely two solutions, equal to
$\pm\zeta_n^+\big(-i(m+\frac{1}{2})\big)$, for each
$n\in\{0,\dots,N\}$ which is different from $m$. Hence the fiber of $-i(m+\frac{1}{2})$ contains exactly $2^N$ points. These points have a real coordinate
$\zeta_n$ when  $N\geq n> m$, a purely imaginary coordinate $\zeta_n$ when $0\leq n <m$, and $\zeta_m=0$. They are the branching points of $\M_{(N)}$.
A schematic picture of the coordinates of the points of the fiber above $-i\R^+$ is drawn in Figure 2.

\begin{center}
\begin{figure}
\vskip -3.5truecm

\null\qquad\includegraphics[width=16cm]{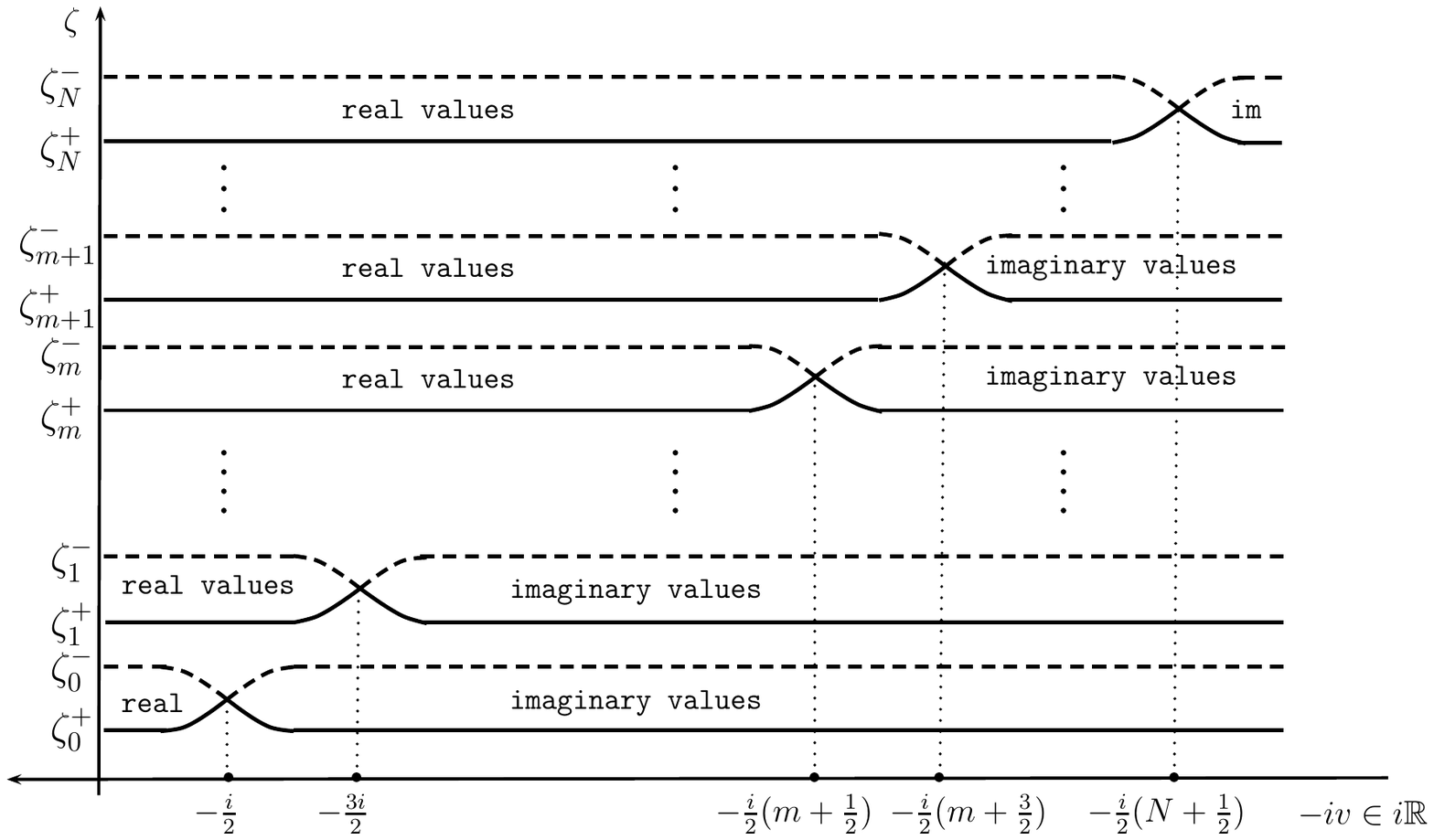}

\vskip -13truecm

\caption{Schema of the fibre above  $-i\R^+$ in $\M_{(N)}$}
\end{figure}
\end{center}

We want to construct a meromorphic lift of $F$ along the branched curve
\begin{equation}
\label{eq:gamma}
\gamma_{N}: v \in \big(0,N+\tfrac{3}{2}\big) \mapsto (-iv,\pm\zeta_0^+(-iv), \dots, \pm\zeta_N^+(-iv)) \in \M_{(N)}\,,
\end{equation}
which is the lift of $-i\big(0,N+\tfrac{3}{2}\big)$ in $\M_{(N)}$.
More precisely, let $W_{(m)}$ be the open set in $\C$ defined in (\ref{eq:Wm}) and (\ref{eq:Wminusone}).
Then we meromorphically lift $F$ to the open neighborhood
\begin{equation}
\label{eq:MNgamma}
\M_{(\gamma_N)}=\pi_{(N)}^{-1}\Big(\bigcup_{m=-1}^N W_{(m)}\Big)
\end{equation}
of $\gamma_N$ in $\M_{(N)}$.

\smallskip

Notice that by (\ref{eq:assWvbd}) the radius $R_m$ of the open disk $W_{m+\frac{1}{2}}$ satisfies $R_m<1/2$. Moreover, using also (\ref{eq:assWvint}), we have for $m\in \{0,1,\dots,N-1\}$
\begin{eqnarray*}
&&W_{(m)}\cap W_{(m+1)}=W_{(m)}\cap W_{m+\tfrac{3}{2}}=(W_{(m)}\setminus W_{m+\tfrac{1}{2}})\cap W_{m+\tfrac{3}{2}}\,,\\
&&W_{(m)}\cap W_{(m+1)} \cap i\R=-i\big]m+\tfrac{3}{2}-R_{m+1},m+\tfrac{3}{2}\big[\,.
\end{eqnarray*}
Hence, for $m\in \{-1,0,\dots,N\}$,
\begin{equation}
\label{eq:propWm1}
W_{(m)}\cap W_{(m+1)} \cap i\big(\mathbb Z+\tfrac{1}{2}\big) =\emptyset
\end{equation}
and the branching point $-i\big(m+\frac{3}{2}\big)$ is a boundary point of $W_{(m)}\cap W_{(m+1)}$.

For every $\eps=(\eps_0,\dots,\eps_N)\in \{\pm 1\}^{N+1}$ we define a section
$$\sigma_\eps: \C\setminus
i\big(\big(-\infty,- \tfrac{1}{2}\big] \cup \big[\tfrac{1}{2},+\infty\big)\big)\to \M_{(N)}$$
of $\pi_{(N)}$ by setting
\begin{equation}
\label{eq:sigmaeps}
\sigma_\eps(z)=(z,\eps_0\zeta_0^+(z),\dots,\eps_N\zeta_N^+(z))\,.
\end{equation}
This is well defined as $\zeta_n^+(z)$ is well defined at $z\in \C\setminus
i\big(\big(-\infty,-(n+\frac{1}{2})\big)\big] \cup \big[n+\frac{1}{2},+\infty\big)\big)$ for all $n\in \{0,1,\dots,N\}$.

Because of (\ref{eq:propWm1}), by possibly shrinking the open disks $W_v$, we can also assume
that $\pi_{(N)}^{-1}(W'_{(m)})$ is the disjoint union of $2^{N+1}$ homeomorphic copies of $W'_{(m)}$.
In particular, each of these copies is a connected set. We denote by $U_{m,\eps}$ the copy containing
$\sigma_\eps(W'_{(m)}\setminus i\R)$.

Similarly, by possibly shrinking the open disks $W_{m+\tfrac{1}{2}}$, we can furthermore assume
that their preimage $\pi_{(N)}^{-1}\big(W_{m+\tfrac{1}{2}}\big)$ in $\M_{(N)}$ is the disjoint union of $2^{N}$ homeomorphic copies
of $\pi_m^{-1}(W_{m+\tfrac{1}{2}})$. They can be parameterized by the elements
\begin{equation}
\label{eq:epsmcheck}
(\eps_0,\dots, \eps_{m-1},\eps_{m+1},\dots,\eps_N) \qquad (\eps_n\in \{\pm 1\}, 0\leq n\leq N, n\neq m)\,.
\end{equation}
We denote by $\eps({m}^\vee)$ the element $(\ref{eq:epsmcheck})$ obtained from $\eps \in \{\pm 1\}^{N+1}$ by removing its $m$-th component. This means that $\eps({m}^\vee)=\eps'({m}^\vee)$ if and only if $\eps$ and $\eps'$ are equal but for their $m$-th component, which can be $\pm 1$. Modulo this identification, we can indicate the connected components of $\pi_{(N)}^{-1}\big(W_{m+\tfrac{1}{2}}\big)$ as $U_{\eps({m}^\vee)}$ with $\eps \in  \{\pm 1\}^{N+1}$.
To unify notation, we define $U_{\eps({m}^\vee)}=\emptyset$ for $m=-1$.

Observe that $\big\{U_{\eps({m}^\vee)}\cup U_{m,\eps}; \eps\in \{\pm 1\}^{N+1}, m\in \mathbb Z\,, -1\leq m\leq N\big\}$
is a covering of $\M_{(\gamma_N)}$ consisting of open connected sets.

\begin{thm}
\label{thm:meroliftF}
For $m\in \{-1,0,\dots,N\}$, $\eps \in \{\pm 1\}^{N+1}$ and $(z,\zeta) \in U_{\eps({m}^\vee)}\cup U_{m,\eps}$ define
\begin{equation}
\label{eq:widetildeF}
\wt F(z,\zeta)=
\displaystyle{F_{(m)}(z)+4\pi i \sum_{n=0}^m\wt G_{(n)}(z,\zeta_n)+4\pi i
\sum_{\stackrel{m<n\leq N}{\text{with $\eps_n=-1$}}} \big[\wt G_{(n)}(z,\zeta_{n})-\wt G_{(n)}(z,-\zeta_{n}) \big]}\,,
 \end{equation}
where the first sum is equal to $0$ if $m=-1$ and the second sum is $0$ if $\eps_n=1$ for all $n>m$.

Then $\wt F$ is a meromorphic lift of $F$ to the open neighborhood
$\M_{(\gamma_N)}$ of the branched curve $\gamma_{N}$
lifting $-i\big(0,N+\tfrac{3}{2}\big)$ in $\M_{(N)}$.
The singularities of $\wt F$ on $\M_{(\gamma_N)}$ are simple poles at the points
$\big(-i\big(m+\frac{1}{2}\big),\zeta\big)\in \M_{(N)}$ with $m\in \{0,1,\dots,N\}$.
\end{thm}
\begin{prf}

To simplify the notation, we write the coordinates of the points of $\M_{(N)}$ above some
fixed $z\in \C$ as $\zeta_n^\pm$ instead of $\zeta_n^\pm(z)$.

Let $\wt F_{(m,\eps)}(z,\zeta)$ denote the right-hand side of (\ref{eq:widetildeF}).
Since for every $m\in \{-1,0,\dots,N-1\}$ the function $\wt F_{(m,\eps)}(z,\zeta)$ is meromorphic on
$U_{\eps({m}^\vee)}\cup U_{m,\eps}$, then $\wt F(z,\zeta)$ will be meromorphic on $\M_{(\gamma_N)}$ provided
 $\wt F_{(m,\eps)}(z,\zeta)=\wt F_{(m',\eps')}(z,\zeta)$ on all nonempty intersections $(U_{\eps({m}^\vee)}\cup U_{m,\eps})\cap (U_{\eps'((m')^\vee)}\cup U_{m',\eps'})$ with $m,m'\in\{-1,0,\dots,N\}$, $\eps,\eps'\in \{\pm 1\}^{N+1}$ and $(m,\eps)\neq (m',\eps')$.

If $(m,\eps)\neq (m',\eps')$, then $U_{m,\eps}\cap U_{m',\eps'}=\emptyset$. Hence
$$
(U_{\eps({m}^\vee)}\cup U_{m,\eps})\cap (U_{\eps'((m')^\vee)}\cup U_{m',\eps'})
=
\big(U_{\eps({m}^\vee)}\cap U_{\eps'((m')^\vee)}\big) \cup \big(U_{\eps({m}^\vee)}\cap U_{m',\eps'}\big)\cup
\big(U_{\eps'((m')^\vee)} \cap U_{m,\eps}\big)\,.$$

Since different $U_{\eps({m}^\vee)}$'s are disjoint, $U_{\eps({m}^\vee)}\cap U_{\eps'((m'^\vee)}\neq \emptyset$ means that $U_{\eps({m}^\vee)}= U_{\eps'((m')^\vee)}$. Since $(m,\eps)\neq (m',\eps')$, this means that $m=m'$,
$\eps_n=\eps'_n$ for all $n\in \{0,1,\dots,N\}$ with $n\neq m$ and $\eps_m=-\eps'_m$. The definition of  $\wt F_{(m,\eps)}(z,\zeta)$ on the right-hand side of (\ref{eq:widetildeF}) depends on $\eps$ only for the $\eps_n$'s with $n>m=m'$. Therefore, in this case we have $\wt F_{(m,\eps)}(z,\zeta)=\wt F_{(m',\eps')}(z,\zeta)$ for $(z,\zeta) \in U_{\eps({m}^\vee)}=U_{\eps'((m')^\vee)}$.

Suppose now that $U_{\eps'((m')^\vee)} \cap U_{m,\eps}\neq \emptyset$. If $m=m'$, this means that $\eps(m^\vee)=
\eps'((m')^\vee)$. We are therefore in the situation just considered, in which $\wt F_{(m,\eps)}(z,\zeta)=\wt F_{(m',\eps')}(z,\zeta)$ for $(z,\zeta) \in U_{\eps(m^\vee)}=U_{\eps'((m')^\vee)}$. If $m\neq m'$, then two cases may occur:
\begin{enumerate}
\item
$m'=m+1$ and $\eps'((m')^\vee)=\eps'((m+1)^\vee)=\eps((m+1)^\vee)$.
\item
$m'=m-1$ and $\eps'((m')^\vee)=\eps'((m-1)^\vee)=\eps((m-1)^\vee)$.
\end{enumerate}
In Case (1) we need to check that $\wt F_{(m,\eps)}(z,\zeta)=\wt F_{(m+1,\eps)}(z,\zeta)$ for $(z,\zeta) \in U_{m,\eps}\cap U_{\eps((m+1)^\vee)}$.
Since $U_{m,\eps}\cap U_{\eps((m+1)^\vee)}$ is connected (as homeomorphic to $W_{(m)}\cap W_{m+\frac{3}{2}}$), it suffices to
check this equality on its subset $\sigma_\eps\big((W_{(m)}\cap W_{m+\frac{3}{2}}) \setminus i\R\big)$. On this subset, we have $\zeta_{m+1}=\eps_{m+1}\zeta_{m+1}^+$. Hence, by (\ref{eq:wtG zetan+}),
\begin{equation*}
\wt G_{(m+1)}(z,\eps_{m+1}\zeta_{m+1})=\wt G_{(m+1)}(z,\zeta_{m+1}^+)=G_{(m+1)}(z)\qquad \big((z,\zeta)\in \sigma_\eps\big((W_{(m)}\cap W_{m+\frac{3}{2}}) \setminus i\R\big)\,.
\end{equation*}
Recall that, by (\ref{eq:FmFm-1}),
$$F_{(m)}(z)=F_{(m+1)}(z)+4\pi i \, G_{(m+1)}(z) \qquad (z \in W_{(m)}\cap W_{m+\frac{3}{2}}).$$
These equalities will be subsequently used in the computations below.

We now need to distinguish two further cases inside Case (1):
\begin{enumerate}
\renewcommand{\theenumi}{\alph{enumi}}
\renewcommand{\labelenumi}{{\rm (1.\theenumi)}}
\item
$\eps_n=1$ for all $n>m$,
\item
there exists $n>m$ such that $\eps_n=-1$.
\end{enumerate}
In Case (1.a), we have $\eps_{m+1}=1$. In Case (1.b), we also have to distinguish whether $\eps_{m+1}=1$ or $\eps_{m+1}=-1$.
In Case (1.a) or in Case (1.b) with $\eps_{m+1}=1$,
for $(z,\zeta) \in \sigma_\eps\big((W_{(m)}\cap W_{m+\frac{3}{2}}) \setminus i\R\big)$
(and with empty sums equal to 0), we have:
\begin{eqnarray*}
\wt F_{(m,\eps)}(z,\zeta)&=&F_{(m)}(z)+4\pi i\sum_{n=0}^{m}\wt G_{(n)}(z,\zeta_n)
+4\pi i\sum_{\stackrel{m+1<n\leq N}{\text{with $\eps_n=-1$}}} \big[\wt G_{(n)}(z,\zeta_{n})-\wt G_{(n)}(z,-\zeta_{n}) \big]\\
&=&F_{(m+1)}(z)+ 4\pi i \, G_{(m+1)}(z)+4\pi i\sum_{n=0}^{m}\wt G_{(n)}(z,\zeta_n)\\
&&\qquad \qquad \qquad
+4\pi i\sum_{\stackrel{m+1<n\leq N}{\text{with $\eps_n=-1$}}} \big[\wt G_{(n)}(z,\zeta_{n})-\wt G_{(n)}(z,-\zeta_{n}) \big]\\
&=&F_{(m+1)}(z)+4\pi i\sum_{n=0}^{m+1}\wt G_{(n)}(z,\zeta_n)
+4\pi i\sum_{\stackrel{m+1<n\leq N}{\text{with $\eps_n=-1$}}} \big[\wt G_{(n)}(z,\zeta_{n})-\wt G_{(n)}(z,-\zeta_{n}) \big]\\
&=&\wt F_{(m+1,\eps)}(z,\zeta)\,.
\end{eqnarray*}
In Case (1.b) with $\eps_{m+1}=-1$, then for $(z,\zeta) \in \sigma_\eps\big((W_{(m)}\cap W_{m+\frac{3}{2}}) \setminus i\R\big)$
(and with empty sums equal to 0), we have:
\begin{eqnarray*}
\wt F_{(m,\eps)}(z,\zeta)&=&F_{(m)}(z)+4\pi i\sum_{n=0}^{m}\wt G_{(n)}(z,\zeta_n)
+4\pi i\;\big[\wt G_{(m+1)}(z,\zeta_{m+1})-\wt G_{(m+1)}(z,-\zeta_{m+1}) \big]\\
&&\qquad \qquad \qquad
+4\pi i\sum_{\stackrel{m+1<n\leq N}{\text{with $\eps_n=-1$}}} \big[\wt G_{(n)}(z,\zeta_{n})-\wt G_{(n)}(z,-\zeta_{n}) \big]\\
&=&F_{(m+1)}(z)+ 4\pi i \, G_{(m+1)}(z)+4\pi i\sum_{n=0}^{m}\wt G_{(n)}(z,\zeta_n)\\
&&\qquad \qquad \qquad
+4\pi i\; \big[\wt G_{(m+1)}(z,\zeta_{m+1})-\wt G_{(m+1)}(z,-\zeta_{m+1}) \big]\\
&&\qquad \qquad \qquad \qquad \qquad \qquad
+4\pi i\sum_{\stackrel{m+1<n\leq N}{\text{with $\eps_n=-1$}}} \big[\wt G_{(n)}(z,\zeta_{n})-\wt G_{(n)}(z,-\zeta_{n}) \big]\\
&=&F_{(m+1)}(z)+4\pi i\sum_{n=0}^{m+1}\wt G_{(n)}(z,\zeta_n)
+4\pi i\sum_{\stackrel{m+1<n\leq N}{\text{with $\eps_n=-1$}}} \big[\wt G_{(n)}(z,\zeta_{n})-\wt G_{(n)}(z,-\zeta_{n}) \big]\\
&=&\wt F_{(m+1,\eps)}(z,\zeta)\,.
\end{eqnarray*}
This concludes Case (1).

For Case (2), we need to check that $\wt F_{(m-1,\eps)}(z,\zeta)=\wt F_{(m,\eps)}(z,\zeta)$ for $(z,\zeta) \in U_{m,\eps}\cap U_{\eps((m-1)^\vee)}$. As in Case (1), since $U_{m,\eps}\cap U_{\eps((m-1)^\vee)}$ is connected (as homeomorphic to $W_{(m)}\cap W_{m+\frac{1}{2}}$), it suffices to
check this equality on $\sigma_\eps\big((W_{(m)}\cap W_{m+\frac{1}{2}}) \setminus i\R\big)$, where $\zeta_{m}=\eps_{m}\zeta_{m}^+$. The proof parallels that of Case (1), by considering if $\eps_n=1$ for all $n>m-1$ (Case (2.a)) or if there exists $n>m-1$ such that $\eps_n=-1$ (Case (2.b)). In Case (2.b), one moreover has to distnguish whether
$\eps_m=1$ or $\eps_m=-1$. The details are omitted.

This concludes the proof that $\wt F(z,\zeta)$ is a meromorphic function on $\M_{(\gamma_N)}$. Its singularities on $W_{(m)}$ are those
of $\wt G_{(m)}(z,\zeta_{m})$, i.e. simple poles at the branching points above $z=-i\big(m+\frac{1}{2})$, with $m\in \{0,1,\dots,N\}$.

The fact that $\wt F(z,\zeta)$ is a lift of $F$ is an immediate consequence of (\ref{eq:wtG zetan+}), Corollary \ref{cor:FextendedMainSection} and the definition of $\wt F_{(m,\eps)}(z,\zeta)$ for $\eps=(1,\dots,1)$.
\end{prf}

\subsection{Meromorphic extension of the resolvent}

\label{subsection:mero-ext-resolvent}

Let $\C^-=\{z\in \C;\Im z<0\}$ be the lower half plane. In this subsection we meromorphically extend the resolvent
$z \mapsto [R(z)f](y)$ (where $f\in C_c^\infty(\X)$ and $y\in \X$ are arbitrarily fixed) from $\C^- \setminus i(-\infty,-\frac{1}{2}\rhoX]$ across the half-line  $ i(-\infty,-\frac{1}{2}\rhoX]$.  As before, we shall omit the dependence on $f$ and $y$ from the notation and write $R(z)$ instead of $[R(z)f](y)$. This simplification of notation will be employed wherever it is appropriate.

Recall from Proposition \ref{pro:holoextRF}  that on $\C^-$ (and indeed on a larger domain) we can write
\begin{equation}
\label{eq:holoextRF-mod}
R(z)=H(\rhoX^{-1}z)+\frac{\pi i}{|W| \rhoXtwo}\, F(\rhoX^{-1}z) \qquad (z\in \C^-),
\end{equation}
where $H$ is a holomorphic function. The required meromorphic extension will be therefore deduced from that
of $F$ in  Section \ref{subsection:mero-ext-F}.

For any fixed integer $N\geq 0$ consider the Riemann surface above
$$\C^-_{N}=\big\{z \in \C^-;\Im z >-(N+\tfrac{3}{2})\rhoX\big\}$$
defined by
\begin{multline*}
\M_{(\X,N)}
=\left\{(z,\zeta);\ z\in
\C^-_{N}, \, \zeta=(\zeta_0, \zeta_1,\dots, \zeta_N)\in \C^{N+1},\right.\\
\left.  \zeta_n^2=\left(\tfrac{i}{z}\big(n+\tfrac{1}{2}\big)\right)^2-\tfrac{1}{\rhoXtwo}\,, \; 0\leq n\leq N\right\}.
\end{multline*}
Hence $(z,\zeta)\in \M_{(\X,N)}$ if and only if $(\rhoX^{-1}z,\rhoX\zeta)\in \M_{(N)}$.
Then the curve
$$v \in \big(0,N+\tfrac{3}{2}\big)\mapsto -i\rhoX v\in i\R$$
on the negative imaginary axis
lifts to the curve $\gamma_{\X,N}$ on the Riemann surface $\M_{(\X,N)}$ given by
\begin{equation}
\label{eq:gammaXN}
\gamma_{\X,N}(v)=(-i\rhoX v, \rhoX^{-1}\zeta_0^\pm(-iv),\dots,\rhoX^{-1}\zeta_N^\pm(-iv))\qquad (0<v<N+\tfrac{3}{2})\,,
\end{equation}
where $\zeta_n^+(-iv)=-\zeta_n^-(-iv)$ is as in (\ref{eq:zetan+-iv}).

From each object constructed for $\M_{(N)}$ we obtain a corresponding object for $\M_{(\X,N)}$
by replacing $(z,\zeta)$ with $(\rhoX^{-1}z,\rhoX\zeta)$. It will be denoted by adding ``$\X$'' to the symbol used in Section \ref{subsection:mero-ext-F} for the corresponding object
in $\M_{(N)}$. For instance, for $\eps\in \{\pm 1\}^{N+1}$ and $n\in \{-1,0,1,\dots,N\}$, the sets $U_{\X,n,\eps}$ and $U_{\X,\eps(n^\vee)}$ consist of the
$(z,\zeta)\in \M_{(\X,N)}$ so that the points $(\rhoX^{-1}z,\rhoX\zeta)$ belong to  $U_{n,\eps}$ and $U_{\eps(n^\vee)}$, respectively.
Similarly,
$$
\M_{(\gamma_{\X,N})}=\{(z,\zeta)\in \M_{(\X,N)}; (\rhoX^{-1}z,\rhoX\zeta)\in \M_{(\gamma_N)}\}\,.
$$
Also, for $n\in \{-1,0,1,\dots,N\}$,
$$
W_{(\X,n)}=\rhoX^{-1}W_{(n)}\,, \quad W'_{(\X,n)}=\rhoX^{-1}W'_{(n)}\,,
\quad W_{\X,n+\frac{1}{2}}=\rhoX^{-1} W_{n+\frac{1}{2}}\,.
$$
Hence $W_{(\X,n)}=W_{\X,n+\frac{1}{2}}\cup W'_{(\X,n)}$ is an open neighborhood in $\C^-_N$ of the interval $-i\rhoX \big[n+\frac{1}{2},n+\frac{3}{2}\big)$ if $n\neq -1$, and of
$-i\big(0,\frac{\rhoX }{2}\big)$ if $n=-1$.

Let $\pi_{(\X,N)}: \M_{(\X,N)} \to \C_N^-$ denote the projection $\pi_{(\X,N)}((z,\zeta))=z$.
The Riemann surface $\M_{(\X,N)}$ admits branching points at the $2^N$ points above $z^{(n)}=-i\big(n+\frac{1}{2}\big)\rhoX$, where $n\in \{0,1,\dots,N\}$.
The open disk $W_{\X,n+\frac{1}{2}}$ is centered at $z^{(n)}$.
Set
\begin{equation}
\label{eq:Espn}
\mathcal E_n=\{\eps \in \{\pm 1\}^{N+1}; \eps_n=1\}\,.
\end{equation}
Then the sets $U_{\X,\eps(n^\vee)}$ with $\eps\in \mathcal E_n$ are pairwise disjoint and form an open cover of $\pi_{(\X,N)}^{-1}\Big(W_{\X,n+\frac{1}{2}}\Big)$.
For $\eps \in \mathcal E_n$, we denote by $(z^{(n)},\zeta^{(n,\eps)})$ the point of
the fiber  $\pi_{(\X,N)}^{-1}(z^{(n)})$ which belongs to $U_{\X,\eps(n^\vee)}$.
Moreover, the map
\begin{equation}\label{chart-n}
\kappa_{n,\eps}:U_{\X,\eps(n^\vee)} \ni (z,\zeta)\to \zeta_n\in \C\setminus
i\big(\big(-\infty, -\rhoX^{-1}\big] \cup \big[\rhoX^{-1},+\infty\big)\big),\ \
z=-\frac{i\rhoX(n+\frac{1}{2})}{\sqrt{\rhoXtwo\zeta_n^2+1}}
\end{equation}
is a local chart around $(z^{(n)},\zeta^{(n,\eps)})$.

Observe also that,
as in (\ref{eq:MNgamma}),
\begin{equation}
\label{eq:MgammaXN}
\M_{(\gamma_{\X,N})}=\pi_{(\X,N)}^{-1}\Big(\bigcup_{n=-1}^N W_{(\X,n)}\Big)
\end{equation}
is an open neighborhood of the curve $\gamma_{\X,N}$ in $\M_{(\X,N)}$. Moreover, for every
$n\in \{-1,0,\dots,N\}$, we have
$$
\pi_{(\X,N)}^{-1}(W_{(\X,n)})=\bigcup_{\eps\in \{\pm 1\}^{N+1}} (U_{\X,n,\eps} \cup U_{\X,\eps(n^\vee)})\,.
$$
The resolvent $R$ can be lifted along the curve $\gamma_{\X,N}$ to a meromorphic function on $\M_{(\gamma_{\X,N})}$. The meromorphically lifted function admits singularities at the $(N+1)\times 2^N$ branching points $(z^{(n)},\zeta^{(n,\eps)})$ with $(n,\eps)\in \{0,1,\dots,N\}\times \mathcal E_n$. They are simple poles.  The precise situation is given by the following theorem. Recall the notation $\bxi=e^{i\pi/3}$ and the functions $\psi_z$ and $\phi_{z,u}$ from (\ref{phiwu}) and (\ref{psiwu}).

\begin{thm}
\label{thm:meroextshiftedLaplacian}
Let $f \in C^\infty_c(\X)$ and $y \in \X$ be fixed. Let $N\in \mathbb N$ and let $\gamma_{\X,N}$ be the curve in $\M_{(\X,N)}$
given by (\ref{eq:gammaXN}). Then the resolvent $R(z)=[R(z)f](y)$ lifts as a meromorphic function to the neighborhood $\M_{(\gamma_{\X,N})}$ of the curve $\gamma_{\X,N}$ in $\M_{(\X,N)}$. We denote the lifted meromorphic function by
$\wt R_{(N)}(z,\zeta)=\big[\wt R_{(N)}(z,\zeta)f\big](y)$.

The singularities of $\wt R_{(N)}$ are simple poles at the points $(z^{(n)},\zeta^{(n,\eps)})\in \M_{(\X,N)}$ with $z^{(n)}=-i(n+\frac{1}{2})\rhoX$ and $(n,\eps)\in \{0,1,\dots,N\}\times \mathcal E_n$. Explicitly, for $(n,\eps)\in \{0,1,\dots,N\}\times \mathcal E_n$,
\begin{eqnarray}
\label{eq:RNnearwn}
\wt R_{(N)}(z,\zeta)=\wt H_{(N,n,\eps)}(z,\zeta)+C \,\wt G_{(\X,n)}(z,\zeta) \qquad
((z,\zeta)\in  U_{\X,n,\eps} \cup U_{\X,\eps(n^\vee)})
\,,
\end{eqnarray}
where $\wt H_{(N,n,\eps)}$ is holomorphic,
$$
C=\frac{12\pi^2}{\rhoXtwo |W|}=\frac{\pi^2}{|W|}
$$
is a constant (independent of $N$, $n$ and $\eps$), and
\begin{eqnarray}\label{eq:wtGnXn}
\wt G_{(\X,n)}(z,\zeta)=
\psi_{z\rhoX^{-1}}\Big(\rhoX\big(\tfrac{i}{z}(n+\tfrac{1}{2})-\zeta_n\big)\Big)
\phi_{z\rhoX^{-1},1}\big(\bxi\rhoX\big(\tfrac{i}{z}(n+\tfrac{1}{2})-\zeta_n\big)\big) \\
\cdot \,
\phi_{z\rhoX^{-1},1}\big(\bxi^2\rhoX \big(\tfrac{i}{z}(n+\tfrac{1}{2})-\zeta_n\big)\big) \frac{\tfrac{i}{z}(n+\tfrac{1}{2})}{-i\pi\zeta_n} \nn
\end{eqnarray}
is independent of $N$ and $\eps$ (but dependent on $f$ and $y$, which are omitted from the notation).
The function $\wt G_{(\X,n)}$ has a simple pole at $(z^{(n)},\zeta^{(n,\eps)})$ for all $\eps\in \mathcal E_n$.

The local expression for $\wt G_{(\X,n)}$ in terms of the chart (\ref{chart-n}) is:
\begin{multline}\label{eq:Gn-cover}
(\wt G_{(\X,n)}\circ \kappa_{n,\eps}^{-1})(\zeta_n)=
\psi_{\frac{i(n+\frac{1}{2})}{\sqrt{\rhoXtwo\zeta_n^2+1}}}
\Big(\sqrt{\rhoXtwo\zeta_n^2+1}+\rhoX\zeta_n\Big)\\
\cdot \phi_{\frac{i(n+\frac{1}{2})}{\sqrt{\rhoXtwo\zeta_n^2+1}},1}
\Big(\bxi\big(\sqrt{\rhoXtwo\zeta_n^2+1}+\rhoX\zeta_n\big)\Big) \\
\cdot \phi_{\frac{i(n+\frac{1}{2})}{\sqrt{\rhoXtwo\zeta_n^2+1}},1}
\Big(\bxi^2\big(\sqrt{\rhoXtwo\zeta_n^2+1}+\rhoX\zeta_n\big)\Big) \cdot
\frac{\sqrt{\rhoXtwo\zeta_n^2+1}}{i\pi\rhoX\zeta_n}.
\end{multline}
Furthermore,
\begin{eqnarray}
 \label{eq:residueGn-cover}
\Res_{\zeta_n=0} \big[(\wt R_{(n)} \circ \kappa_{n,\eps}^{-1})(\zeta_n)f](y)
&=& -\frac{1}{4|W|}
\Big(n+\frac{1}{2}\Big)^2 \big(f \times \varphi_{(n+\frac{1}{2})\rho}\big)(y)\,,
\end{eqnarray}
where
$\varphi_{(n+\frac{1}{2})\rho}$ is the spherical function on $\X$
of spectral parameter $(n+\frac{1}{2})\rho$, see (\ref{eq:phil}).
In particular, $\Res_{\zeta_n=0} \big[(\wt R_{(n)} \circ \kappa_{n,\eps}^{-1})(\zeta_n)f](y)$ is independent of $\eps$.
\end{thm}

\begin{proof}
According to Corollary \ref{cor:logRiemann} and Proposition \ref{pro:holoextRF}, it suffices to meromorphically extend the
function $F$ given by (\ref{eq:F}), as done in the Section \ref{subsection:mero-ext-F}.
The lift $\wt R_{(N)}$ of the resolvent  to $\M_{(\gamma_{\X,N})}$ as well as the expression of  $\wt G_{(\X,n)}$ are then obtained from Theorem \ref{thm:meroliftF},
with $(z,\zeta)$ replaced by $(\rhoX^{-1} z,\rhoX\zeta)$.
In fact, the function $\wt H_{(N,n,\eps)}(z,\zeta)$ in (\ref{eq:RNnearwn}) is the sum of the holomorphic function $H(\rhoX^{-1} z)$ from Proposition \ref{pro:holoextRF} together with  $\frac{\pi i}{|W| \rhoXtwo}$ times
the sum of all terms in the definition of $\wt F(\rhoX^{-1} z,\rhoX \zeta)$ on $U_{\X,\eps(n^\vee)}\cup U_{\X,n,\eps}\,$, as
on the right-hand side of (\ref{eq:widetildeF}) with $n$ instead of $m$, except for $4\pi i \, \wt G_{(n)}(\rhoX^{-1} z,\rhoX \zeta)$. Thus $\wt H_{(N,n,\eps)}$ is holomorphic on $U_{\X,\eps(n^\vee)}\cup U_{\X,n,\eps}\,$.
Finally,
\begin{equation}
\label{eq:GntGn}
\wt G_{(\X,n)}(z,\zeta)=-\tfrac{1}{3}\,\wt G_{(n)}(\rhoX^{-1} z,\rhoX\zeta_n)\,.
\end{equation}

Notice that for $(z,\zeta)$ in the neighborhood $U_{\X,\eps(n^\vee)}$  of
$(z^{(n)},\zeta^{(n,\eps)})$ we have
$$
i \frac{\rhoX}{z}\big(n+\tfrac{1}{2}\big)=-\sqrt{\rhoXtwo\zeta_n^2+1}\,.
$$
The expression (\ref{eq:Gn-cover}) for $\wt G_{(\X,n)}$ in the chart $(U_{\X,\eps(n^\vee)},\kappa_{n,\eps})$ comes then from (\ref{eq:GntGn}) together with the parity properties (\ref{even phiwu}) and  (\ref{even phiwu bis})  of the functions $\psi_z$ and $\phi_{z,1}$, as for (\ref{expression in terms of charts 2}).

To compute the residue (\ref{eq:residueGn-cover}), we have,
by (\ref{eq:RNnearwn}),
\begin{equation}
\label{eq:forresidue1}
\Res_{\zeta_n=0} (\wt R_{(N)} \circ \kappa_{n,\eps}^{-1})(\zeta_n)=C \, \Res_{\zeta_n=0} (\wt G_{(\X,n)}\circ\kappa_{n,\eps}^{-1})(\zeta_n) \,.
\end{equation}
By (\ref{eq:residue-n-minus}) and (\ref{eq:GntGn}),
\begin{equation}
\label{eq:forresidue2}
\Res_{\zeta_n=0} \;(\wt G_{(\X,n)}\circ\kappa_{n,\eps}^{-1})(\zeta_n)=
\frac{1}{2\pi} \Big(n+\frac{1}{2}\Big)^2
\psi_{i(n+\tfrac{1}{2})}(1)\,.
\end{equation}
Moreover, (\ref{psiwu}) yields
$$
\psi_{i(n+\tfrac{1}{2})}(1)=\frac{1}{2\pi}\,\Big(f\times
\frac{\varphi_{\compl \l(i(n+\frac{1}{2}),1)}+\varphi_{-\compl\l(i(n+\frac{1}{2}),1)}}{2} \Big)(y)\,.
$$
By definition, see (\ref{eq:lzw}), we have $\l(i(n+\tfrac{1}{2}),1)=\compl \l(n+\tfrac{1}{2},1)$. Furthermore, by (\ref{eq:lpolar})
and (\ref{eq:alpha-e}),
$$
\l(n+\tfrac{1}{2},1)=(n+\tfrac{1}{2})e_1=(n+\tfrac{1}{2})\alpha_{1,2}\,.
$$
Recall that the spherical functions $\varphi_\l$ are Weyl group invariant in the spectral parameter $\l\in\mathfrak{a}_\C^*$. Recall also that the Weyl group $\Wg$ permutes all roots, and more precisely it acts on $\alpha_{i,j}$ by permuting the indices $i,j \in \{1,2,3\}$.
Since $\rho=\alpha_{1,3}$ is a root,
we conclude that
$$
\varphi_{\pm \compl \l(i(n+\frac{1}{2}),1)}=\varphi_{\pm \compl^2 \l((n+\frac{1}{2}),1)}=
\varphi_{\mp (n+\frac{1}{2})\alpha_{1,2}}=\varphi_{(n+\frac{1}{2})\rho}\,.$$
Therefore
\begin{equation}
\label{eq:forresidue3}
\psi_{i(n+\tfrac{1}{2})}(1)=
-\,\frac{1}{2\pi} \big(f \times \varphi_{(n+\frac{1}{2})\rho}\big)(y)\,.
\end{equation}
The residue (\ref{eq:residueGn-cover}) follows then from (\ref{eq:forresidue1}), (\ref{eq:forresidue2}) and (\ref{eq:forresidue3}).
\end{proof}

\section{The residue operators}\label{section:residues}

Theorem \ref{thm:meroextshiftedLaplacian} gives the meromorphic lift of the resolvent
of the (shifted) positive Laplace-Beltrami operator $\Delta-\inner{\rho}{\rho}$ of $\X$ to the Riemann surface $\M_{(\X,N)}$ covering $\C^-_N=\{z \in \C^-: \Im z >-(N+\frac{3}{2})\rhoX\}$, where $N$ be any fixed nonnegative integer. For fixed $f\in C^\infty_c(\X)$ and $y \in \X$, the lifted resolvent $\wt R_{(N)}$
admits simple poles at the branching points of $\M_{(\X,N)}$, that is at the points
$(z^{(n)},\zeta^{(n,\eps)})$ with $(n,\eps)\in \{0,1,\dots,N\}\times \{\pm 1\}^{N+1}$.
The singular part  of the function $(z,\zeta) \to [\wt R_{(N)}(z,\zeta)f](y)$ at $(z^{(n)},\zeta^{(n,\eps)})$ is a constant multiple the function $\wt G_{(\X,n)}(z,\zeta)$ defined in (\ref{eq:GntGn}). It is independent of the choice of
$N\geq n$ and of the singular point $(z^{(n)},\zeta^{(n,\eps)})$ in the fiber above $z^{(n)}=-i\rhoX(n+\frac{1}{2})$ in $\M_{(\X,N)}$.
In terms of the canonical chart $(U_{\X,\eps(n^\vee)},\kappa_{n,\eps})$ around $(z^{(n)},\zeta^{(n,\eps)})$, the residue of $\wt R_{(N)}(z,\zeta)$ at $(z^{(n)},\zeta^{(n,\eps)})$, computed in (\ref{eq:residueGn-cover}), is a smooth function of $y \in X$ depending linearly on $f \in C^\infty_c(X)$. We use the value of this residue to define a residue operator at $z^{(n)}=-i\rhoX(n+\frac{1}{2})$. More precisely, we define the residue operator
\begin{equation}
\label{eq:ResnR}
\Resn: C^\infty_c(\X) \to C^\infty(\X)
\end{equation}
by
\begin{equation}
\label{eq:ResnRf}
\Resn f= \Big(n+\frac{1}{2}\Big)^2 \big(f \times \varphi_{(n+\frac{1}{2})\rho}\big) \qquad (f \in C^\infty_c(\X))\,.
\end{equation}

In this section we study the range of the linear operator $\Resn$ from a representation theoretic point of view.

\subsection{Residue operators and eigenspace representations}
\label{subsection:residue-eigenspace}

Let $\l\in \fraka_\C^*$. We consider the convolution operator
\begin{equation}
\label{eq:convolution-lambda}
\mathcal R_\lambda: C_c^\infty(\X) \to C^\infty(\X)
\end{equation}
defined by
\begin{equation}
\label{eq:convolution-lambdaf}
\mathcal R_\lambda(f)=f \times \varphi_\lambda \qquad  (f \in C^\infty_c(\X))\,,
\end{equation}
where, as before, $\varphi_\lambda$ denotes
Harish-Chandra's spherical function of spectral parameter $\lambda$. We keep the notation on
eigenspace representations introduced in Section \ref{subsection:eigenspace-reps}.

The next proposition holds for arbitrary Riemannian symmetric spaces of the noncompact type $\X=\G/\K$, where $\G$ is a noncompact connected semisimple Lie group with finite center and $\K$ a maximal compact subgroup of $\G$.  It characterizes the closure of $\mathcal R_\lambda(C^\infty_c(\X))$ inside the eigenspace representation space
$\mathcal E_\l(\X)$ and gives a necessary and sufficient condition for $\mathcal R_\lambda(C^\infty_c(\X))$ to be finite dimensional.

\begin{pro}
\label{prop:eigenspaces}
The space $\mathcal R_\lambda(C^\infty_c(\X))=\{f \times \varphi_\lambda: f \in C^\infty_c(\X)\}$ is a non-zero
$T_\lambda$-invariant subspace of $\mathcal E_\lambda(\X)$. Its closure is the unique closed  irreducible subspace $\mathcal E_{(\lambda)}(\X)$ of $\mathcal E_\lambda(\X)$, which is generated by the translates of the spherical function $\varphi_\lambda$. The space $\mathcal R_\lambda(C^\infty_c(\X))$ is finite dimensional if and only if
$\mathcal E_{\lambda,\G}(\X)\neq \{0\}$ is the finite dimensional spherical representation of highest
restricted weight $-w\lambda-\rho$ (for some $w \in \W$, the Weyl group). In the latter case, $\mathcal R_\lambda(C^\infty_c(\X))=\mathcal E_{\lambda,\G}(\X)$.
\end{pro}
\begin{proof}
Observe first that $\mathcal R_\lambda(C^\infty_c(\X))\neq \{0\}$ as $\varphi_\lambda$ is nonzero and continuous.

For all $D \in \mathbb{D}(\X)$ we have $D(f\times \varphi_\lambda)=f \times
D\varphi_\lambda=\gamma(D)(\lambda)(f \times \varphi_\lambda)$. See \cite[Ch. II, Theorem 5.5]{He2}. So $\mathcal R_\lambda(C^\infty_c(\X))\subseteq \mathcal E_\lambda(\X)$.

Let $g \in G$, and let $F^{\tau(g)}$ denote the left translate by
$g$ of the function $F:\X\to \C$. Hence, if $o=e\K$ is the base point of $\G/\K$, then $F^{\tau(g)}(h\cdot o)=F(g^{-1}h\cdot o)$ for all $h \in \G$. We have
$$T_\lambda(g)(f\times \varphi_\lambda)=(f \times \varphi_\lambda)^{\tau(g)}=f^{\tau(g)} \times \varphi_\lambda\,.$$
As $f^{\tau(g)} \in C_c^\infty(\X)$, the subspace $\mathcal R_\lambda(C^\infty_c(\X))$ of $\mathcal E_\lambda(\X)$ is $T_\lambda$-invariant.

By definition, for $f\in C_c^\infty(\X)$,
$$
(f \times \varphi_\l)(x)=\int_\G f(g\cdot o) \varphi_\l(g^{-1}\cdot x)\, dg=
\int_\G f(g\cdot o)\varphi_\l^{\tau(g)}(x)\, dg \qquad(x\in \X)
$$
belongs to $\mathcal E_{(\lambda)}(\X)$, the closure of the subspace of $\mathcal E_\lambda(\X)$ spanned by the
left translates of $\varphi_\l$. So $\overline{\mathcal R_\lambda(C^\infty_c(\X))}\subseteq \mathcal E_{(\lambda)}(\X)$ is non-zero, closed and $T_\l$-invariant. Since $\mathcal E_{(\lambda)}(\X)$ is irreducible (see e.g. \cite[Ch. IV, Theorem 4.5]{He2}), they must agree.

If $\mathcal R_\lambda(C^\infty_c(\X))$ is finite dimensional, then its nonzero elements are $\G$-finite, so
$\mathcal E_{\lambda,\G}(\X)\neq \{0\}$. By Proposition \ref{prop:sphrep} we conclude that $w\lambda-\rho$ is the restricted highest weight of a finite-dimensional spherical representation for some $w \in \W$. Moreover $\mathcal R_\lambda(C^\infty_c(\X))=\mathcal E_{\lambda,\G}(\X)$ by irreducibility.

Conversely, suppose $\mathcal E_{\lambda,\G}(\X)\neq \{0\}$ is the finite dimensional spherical representation of highest restricted weight $w\lambda-\rho$ for some $w \in \W$. Then
$\mathcal R_\lambda(C^\infty_c(\X))=\mathcal E_{\lambda,\G}(\X)$ by \cite[Theorem 3.2]{HP09}. In particular
$\mathcal R_\lambda(C^\infty_c(\X))$ is finite dimensional.
\end{proof}

\begin{cor}\label{cor:resonances are Poisson-singular}
For all $n \in \mathbb{N}$  the eigenspace representation
$T_{(n+\frac{1}{2})\rho}$ of $\G=\SL(3,\R)$ on $\mathcal E_{(n+\frac{1}{2})\rho}(\X)$ is reducible.
The closure of the image $\Resn (C_c^\infty(\X))$ of the residue operator is the infinite dimensional irreducible  subspace $\mathcal E_{((n+\frac{1}{2})\rho)}(\X)$ of $\mathcal E_{(n+\frac{1}{2})\rho}(\X)$ spanned by the translates of the spherical function $\varphi_{(n+\frac{1}{2})\rho}$. In particular, the residue operator  $\Resn$ has infinite rank for all $n \in \mathbb{N}$.
\end{cor}
\begin{proof}
Since $\rho=\alpha_{1,3}$, we have for all $n \in \mathbb{N}$ that
$$(n+\tfrac{1}{2})\rho_{1,3}=n+\tfrac{1}{2} \in \mathbb{Z}+\tfrac{1}{2}\,.$$
(Recall from Section \ref{subsection:spherical-reps} that $\rho_{i,j}=\rho_{\alpha_{i,j}}$.)
Thus $T_{(n+\frac{1}{2})\rho}$ is reducible by (\ref{eq:gammaX-SL3}).

Because of Proposition \ref{prop:eigenspaces}, it remains to prove that $\Resn (C_c^\infty(\X))$ is infinite
dimensional. For this, it is enough to check that for every Weyl group element $w$ there is a root $\alpha\in \Sigma^+$ so that $-(n+\frac{1}{2}) (w\rho)_\alpha\notin \mathbb{N}+\frac{1}{2}$. In turn, by Remark \ref{rem:hw}, it suffices to check that $(n+\frac{1}{2}) \rho_{i,j} \notin \mathbb{Z}$ when $n\in \mathbb N$ and $(i,j)=(1,2)$ or $(i,j)=(2,3)$. This is immediate, since $\rho_{1,2}=\rho_{2,3}=\frac{1}{2}$.
\end{proof}

\begin{rem}
\label{rem:Poisson}
As before, let $\B=\K/\M$. The Poisson transform of $h\in C(\B)$ is the function $\mathcal P_\l h:\G/\K \to \C$
defined by
\begin{equation*}
(\mathcal P_\l h)(y)=\int_\B h(b)e_{\l,b}(y)\, db\qquad (y\in \G/\K)\,,
\end{equation*}
see e.g. \cite[Ch. II, \S 3, no. 4, and \S 5, no. 4]{He3}.
According to (\ref{eq:conv-with-phil}), the range of the residue operator $\Resn$ is the image under the Poisson tranform of the elements of the Paley-Wiener space $\mathcal H(\fraka_\C^* \times \B)_\W$,
see Section \ref{subsection:spherical-analysis}, evaluated at $\l=(n+\frac{1}{2})\rho$ and considered as a function of $b\in \B$. 
\end{rem}

\subsection{Residue operators and Langlands' classification}
\label{subsection:residues-Langlands}

In this section we give a description of the $\SL(3,\R)$-action on the range of the residue operator in terms of Langlands' classification. We will identify all infinitesimally equivalent representations of the group
$\G=\SL(3,\R)$.

Recall the Iwasawa decomposition $\G=\K\A\N$,
\[
x=\kappa(x)a(x)n(x) \qquad(x\in\G),
\]
from Section \ref{subsection:structureSL3}. Let $\Pg=\M\A\N\subseteq\G$ be the minimal parabolic subgroup consisting of matrices with zeros below the diagonal. For a fixed $\lambda\in\a_\C^*$ the spherical non-unitary principal series representation, $\Ind_\Pg^\G(1\otimes e^{\lambda})$, is a representation of $\G$ defined on the Hilbert space $\mathcal H_\Pi\subseteq\L^2(\K)$ consisting of the right $\M$-invariant functions, with the group action given by
\[
[\Pi(g)v](k)=a(g^{-1}k)^{-\lambda-\rho}v(\kappa(g^{-1}k)) \qquad (g\in\G,\ k\in\K).
\]
This representation has precisely one irreducible subquotient, $\overline \pi(1\otimes(\lambda))(1)$, which contains a trivial $\K$-type, see \cite[Def. 4.4.6]{Vogan81}. If $\Re \lambda$ is negative, then $\overline \pi(1\otimes(\lambda))(1)$ is a subrepresentation of $\Ind_\Pg^\G(1\otimes e^{\lambda})$.

Also, recall that the group $\G$ has only three nilpotent orbits in $\g$ (see e.g. \cite{CMcG}). They are indexed by the partitions of $3$: the maximal orbit $\mathcal O_{1,1,1}$, the minimal orbit $\mathcal O_{2,1}$ and the zero orbit $\mathcal O_{1^3}=\{0\}$.
\begin{pro}\label{The range of the residue operator}
As a representation of $\G$,  the range of the residue operator $\Resn$ is infinitesimally equivalent to $\overline \pi(1\otimes(n+\frac{1}{2}))(1)$. This is a proper infinite dimensional subrepresentation of the non-unitary principal series. This representation  is unitarizable if and only if $n=0$. The wave front set of each of these representations, see \cite{Ross95} for the definition, is equal to $\mathcal O_{2,1}\cup \mathcal O_{1^3}$.
\end{pro}

The proof of Proposition \ref{The range of the residue operator} is based on some well know facts. Since their proofs are short we include them in our argument.

Let $(\Pi,\mathcal H_\Pi)$ be an admissible representation of $\G$ realized on a Hilbert space $\mathcal H_\Pi$ with inner product $(\cdot, \cdot)_\Pi$. The hermitian dual
$(\Pi^h,\mathcal H_\Pi^h)$ is defined by $\mathcal H_\Pi^h=\mathcal H_\Pi$ and $\Pi^h(x)=\Pi(x^{-1})^*$, $x\in \G$.
\begin{lem}\label{lemma:hermitian dual of principal series}
Abusing the notation in an obvious way we have
\[
\left(\Ind_\Pg^\G(1\otimes e^\lambda)\right)^h=\Ind_\Pg^\G(1\otimes e^{-\overline\lambda}).
\]
\end{lem}
\begin{prf}
In the argument below we'll find the following ``change of variables'' formula useful
\begin{equation}\label{equation: change of variables}
\int_\K a(gk)^{-2\rho}f(\kappa(gk))\,dk=\int_\K f(k)\,dk\qquad (g\in\G).
\end{equation}
It may be found for instance in \cite[(7.4)]{knappLie2}.  Let $(\cdot,\cdot)$ denote the $\L^2$ inner product on $\K$.
For $u,v\in \mathcal H_\Pi$ and $g\in\G$ we have
\begin{eqnarray}\label{equation: 1.2}
(\Pi^h(g)u,v)&=&(u,\Pi(g^{-1})v)=
\int_\K u(k)\overline{a(gk)^{-\lambda-\rho}v(\kappa(gk))}\,dk\\
&=&\int_\K a(gk)^{-2\rho}\left(u(k)a(gk)^{-\overline{\lambda}+\rho}\overline{v(\kappa(gk))}\right)\,dk.\nn
\end{eqnarray}
Let $l=\kappa(gk)$. Then $gk=\kappa(gk)a(gk)n(gk)$. So,
\begin{eqnarray*}
k&=&g^{-1}la(gk)n(gk)\\
&=&\kappa(g^{-1}l)a(g^{-1}l)n(g^{-1}l)a(gk)n(gk)\\
&=&\kappa(g^{-1}l)a(g^{-1}l)a(gk)n'.
\end{eqnarray*}
Hence
\[
k=\kappa(g^{-1}l)\ \ \ \text{and}\ \ \ a(g^{-1}l)a(gk)=1.
\]
Let $\psi_g(l)=k$. (This is the inverse of the map $k\to \kappa(gk)$.) The above formulas show that
\[
\psi_g(l)=\kappa(g^{-1}l)\ \ \ \text{and}\ \ \ a(g\psi_g(l))=a(g^{-1}l)^{-1}.
\]
Since the function
\[
f(l)=u(\psi_g(l))a(g\psi_g(l))^{-\overline \lambda +\rho}\overline{v(l)}
\]
is right $\M$-invariant, formula \eqref{equation: change of variables} implies that \eqref{equation: 1.2} is equal to
\begin{eqnarray*}
\int_\K u(\kappa(g^{-1}l))a(g^{-1}l)^{\overline\lambda-\rho}\overline{v(l)}\,dl.
\end{eqnarray*}
Thus
\[
[\Pi^h(g)u](l)=a(g^{-1}l)^{\overline\lambda-\rho}u(\kappa(g^{-1}l)).
\]
Since this is the action of the induced representation $\Ind_\Pg^\G(1\otimes e^{-\overline\lambda})$, we are done.
\end{prf}
Next we show that the range of the residue operator $\Resn$ is infinitesimally equivalent to $\overline \pi(1\otimes \lambda)(1)$, where $\lambda=(n+\frac{1}{2})\rho$.

If $1_\K$ denotes the constant function equal to $1$ on $\K$, then the Harish-Chandra spherical function $\varphi_\lambda$ (considered as a $\K$-bi-invariant function on $\G$) is given by
\[
\varphi_\lambda(x)=(\Pi(x)1_\K,1_\K) \qquad (x\in\G),
\]
where $(\Pi, \mathcal H_\Pi)$ stands for the induced representation $\Ind_\Pg^\G(1\otimes \lambda)$. As our $\lambda$ is real, we have
\[
\varphi_\lambda(x)=(1_\K,\Pi(x)1_\K) \qquad (x\in\G).
\]
Since the convolution of two functions is given by
\[
(f*g)(x)=\int_\G f(y)g(y^{-1}x)\,dy \qquad (x\in\G,\ f\in C_c(\G),\ g\in C(\G)),
\]
we see that
\begin{equation}\label{equation: e0}
(f*\varphi_\lambda)(x)=\int_\G f(y)(1_\K,\Pi(y^{-1}x)1_\K)\,dy=(\Pi^h(f)1_\K,\Pi(x)1_\K).
\end{equation}
Here $\Pi^h$ denotes the lift of $(\Pi^h,\mathcal H_\Pi^h)$ to $L^1(\G)$.
The map
\begin{equation}\label{equation: e1}
C^\infty_c(\G)\ni f\mapsto \Pi^h(f)1_\K\in \mathcal H_\Pi^h
\end{equation}
intertwines the left regular representation on $C^\infty_c(\G)$ with $\Pi^h$.
By definition, the range of this map is generated by the action of
$L^1(\G)$ on the vector $1_\K$. Hence it is generated by the action of the
group on the vector $1_\K$.
But the representation $(\Pi^h,\mathcal H_\Pi^h)$ is infinitesimally equivalent to $\Ind_\Pg^\G(1\otimes (-\lambda))$. As $-\lambda$ is negative, the induced representation contains a unique irreducible subrepresentation, $\overline \pi(1\otimes(-\lambda))(1)$, containing the trivial $\K$-type, \cite[Proposition 4.2.12]{Vogan81}. Hence the range of \eqref{equation: e1} coincides with $\overline \pi(1\otimes(-\lambda))(1)$.

Furthermore we have the map
\begin{equation}\label{equation: e2}
\mathcal H_\Pi^h\ni u\mapsto (u,\Pi(\cdot)1_\K)\in C(\G).
\end{equation}
Since the map
\begin{equation}\label{equation: e3}
C^\infty_c(\G)\ni f\mapsto f*\varphi_\lambda\in C^\infty(\G)
\end{equation}
is the composition of \eqref{equation: e2} and \eqref{equation: e1}, we see that the range of \eqref{equation: e3} coincides with $\overline \pi(1\otimes(-\nu))(1)$. Since replacing $C^\infty_c(\G)$ by $C^\infty_c(\G/\K)$ does not change the range, the range of \eqref{equation: e3} is equal to the range of the residue operator. Since $\overline \pi(1\otimes(-\lambda))(1)$ is isomorphic to $\overline \pi(1\otimes\lambda)(1)$, the first part of the proposition follows.

Next we study some properties of the representation $\overline \pi(1\otimes\lambda)(1)$.
The infinitesimal character of $\overline \pi(1\otimes\lambda)(1)$ is equal to the infinitesimal character of the induced representation, and therefore is represented by $\lambda$, see \cite[Lemma 4.1.8]{Vogan81}. In particular this infinitesimal character is not of the form ``a highest weight plus $\rho$''. Therefore $\overline \pi(1\otimes\lambda)(1)$  is infinitely dimensional.

For each positive root $\alpha$ we have an embedding
\[
\phi_\alpha:\GL(2,\R)\to \G
\]
defined by
\begin{eqnarray*}
\phi_{\alpha_{1,2}}\left(
\begin{array}{cc}
g_{1,1} & g_{1,2}\\
g_{2,1} & g_{2.2}
\end{array}
\right)
&=&
\left(
\begin{array}{ccc}
g_{1,1} & g_{1,2} & 0\\
g_{2,1} & g_{2.2} & 0\\
0 & 0 & (g_{1,1}g_{2.2}- g_{1,2}g_{2,1})^{-1}
\end{array}
\right)\\
\phi_{\alpha_{2,3}}\left(
\begin{array}{ccc}
g_{1,1} & g_{1,2}\\
g_{2,1} & g_{2.2}
\end{array}
\right)
&=&
\left(
\begin{array}{ccc}
(g_{1,1}g_{2.2}- g_{1,2}g_{2,1})^{-1} & 0 & 0\\
0 & g_{1,1}  & g_{1,2}\\
0 & g_{2,1}  & g_{2.2}
\end{array}
\right)\\
\phi_{\alpha_{1,3}}\left(
\begin{array}{cc}
g_{1,1} & g_{1,2}\\
g_{2,1} & g_{2.2}
\end{array}
\right)
&=&
\left(
\begin{array}{ccc}
g_{1,1} & 0 & g_{1,2}\\
0 &  (g_{1,1}g_{2.2}- g_{1,2}g_{2,1})^{-1} & 0\\
g_{2,1} & 0 & g_{2.2}
\end{array}
\right)
\end{eqnarray*}
Then $\phi_\alpha(\GL(2,\R))$ is the centralizer of the kernel of $e^\alpha:\A\to \C$ in $\G$, and is denoted by $\M^\alpha A^\alpha$ in \cite[Notation 4.2.21]{Vogan81}.
As before, let $\lambda=(n+\frac{1}{2})\rho$. We see from \cite[Theorem 4.2.25]{Vogan81} that the induced representation $\Ind_\Pg^\G(1\otimes e^\lambda)$ is reducible if and only if there is $\alpha$ such that the induced representation $\Ind_{\Pg\cap\phi_\alpha(\GL(2,\R))}^{\phi_\alpha(\GL(2,\R))}(1\otimes e^\lambda)$ is reducible. Set
\begin{eqnarray*}
&&H_{\alpha_{1,2}}=
\left(
\begin{array}{ccc}
1 & 0 & 0\\
0 & -1 & 0\\
0 & 0 & 0
\end{array}
\right)\,, \qquad
H_{\alpha_{1,3}}=
\left(
\begin{array}{ccc}
1 & 0 & 0\\
0 & 0 & 0\\
0 & 0 & -1
\end{array}
\right)\,,\qquad
H_{\alpha_{2,3}}=
\left(
\begin{array}{ccc}
0 & 0 & 0\\
0 & 1 & 0\\
0 & 0 & -1
\end{array}
\right).
\end{eqnarray*}
If $\alpha\in\{\alpha_{1,2}, \alpha_{1,3}, \alpha_{2,3}\}$ the reducibility condition for  $\Ind_{\Pg\cap\phi_\alpha(\GL(2,\R))}^{\phi_\alpha(\GL(2,\R))}(1\otimes e^\lambda)$ reads that
\[
\lambda(H_\alpha)=2d+1  \qquad (\text{for some $d=1,2,3,...$}).
\]
But $\rho=\alpha_{1,3}$, so
\[
\rho(H_{\alpha_{1,2}})=1\,, \qquad \rho(H_{\alpha_{1,3}})=2\,, \qquad \rho(H_{\alpha_{2,3}})=1.
\]
Therefore $\Ind_\Pg^\G(1\otimes e^\lambda)$ is reducible if and only if
\begin{eqnarray*}
&&n+\tfrac{1}{2}\ \ \ \text{is an odd positive integer}\\
\text{or}\ \ \ &&2(n+\tfrac{1}{2})\ \ \ \text{is an odd positive integer}\\
\text{or}\ \ \ &&n+\tfrac{1}{2}\ \ \ \text{is an odd positive integer}.
\end{eqnarray*}
Thus $\Ind_\Pg^\G(1\otimes e^\lambda)$ is reducible. According to \cite{speh} the unitary dual of $\G$ consists of the trivial representation, complementary series and unitarily induced representations.
Our representation
$\overline \pi(1\otimes\frac{1}{2}\rho)(1)$ is at the end of a complementary series, hence it is unitarizable. However,  for $n=1,2,3,\dots$,  $\overline \pi(1\otimes(n+\frac{1}{2}\rho))(1)$ is not in any complementary series. Hence it is not unitarizable.

Let $\Qg\subseteq\G$ be the  group generated by $\Pg$ and $\phi_{\alpha_{1,3}}(\GL(2,\R))$. Then $\Qg$ is a maximal parabolic subgroup with the Levi factor equal to
$\phi_{\alpha_{1,3}}(\GL(2,\R))$. The restriction of the character $e^\lambda$ to the center of $\phi_{\alpha_{1,3}}(\GL(2,\R))$ is trivial. Therefore the induction by stages, \cite[Proposition 4.1.18]{Vogan81}, shows that
\begin{equation}\label{stages1}
\Ind_\Pg^\G(1\otimes e^\lambda)=\Ind_\Qg^\G(\Ind_{\phi_{\alpha_{1,3}}(\GL(2,\R))\cap\Pg}^{\phi_{\alpha_{1,3}}(\GL(2,\R))}(1\otimes e^\lambda))\otimes e^0).
\end{equation}
The representation $\Ind_{\phi_{\alpha_{1,3}}(\GL(2,\R))\cap\Pg}^{\phi_{\alpha_{1,3}}(\GL(2,\R))}(1\otimes e^\lambda)$ has the unique Langlands quotient $F_\lambda$, which happens to be finite dimensional. Hence our Langlands quotient $\overline \pi(1\otimes\frac{1}{2}\rho)(1)$ is a subquotient of
\begin{equation}\label{stages2}
\Ind_\Qg^\G(F_\lambda\otimes e^0).
\end{equation}
In particular the wave front set of $\overline \pi(1\otimes\frac{1}{2}\rho)(1)$ is contained in the wave front set of $\Ind_\Qg^\G(F_\lambda\otimes e^0)$, which is equal to the closure of the nilpotent orbit induced from the zero orbit on the Lie algebra of $\phi_{\alpha_{1,3}}(\GL(2,\R))$, i.e. to the closure of $\mathcal O_{2,1}$. This completes the proof of Proposition \ref{The range of the residue operator}.

\begin{rem}\label{rem:Poisson2}
By (\ref{equation: 1.2}), the image of the map (\ref{equation: e2}) is the function
$$
\G \ni g\mapsto (u,\Pi(g)1_\K)=\int_\K u(k) a(g^{-1}k)^{-\l-\rho} \, dk\,.
$$
Consider right $\M$-invariant functions on $\K$ as functions on $\B=\K/\M$ and right $\K$-invariant functions on $\G$ as
functions on $\G/\K$. Then $\mathcal H_\Pi^h=\L^2(\B)$ and the range of (\ref{equation: e2}) is $C^\infty(\G/\K)$. Since
$a(g^{-1}k)^{-\l-\rho}=e_{\l,k\M}(g\K)$, the map (\ref{equation: e2}) is the Poisson transform $\mathcal P_\l$,
see Remark \ref{rem:Poisson}.

For generic $\lambda$, the Poisson transform maps the hyperfunction vectors of spherical non-unitary principal series representation 
$\Ind_\Pg^\G(1\otimes e^\lambda)$ onto $\mathcal E_{\lambda}(\X)$. But Corollary~\ref{cor:resonances are Poisson-singular} says the resonances $(n+\frac{1}{2})\rho$ are not generic in this sense. In this case, according to \cite[Thm. 2.4]{OS}, the image of the Poisson transform consists of those elements $u\in \mathcal E_{(n+\frac{1}{2})\rho}(\X)$ which satisfy the following additional differential equations: $\mathrm{sym}(h)u=0$, where $h$ is a $\K$-harmonic polynomial on $\mathfrak p_\C^*$, viewed as an element of the symmetric algebra $S(\mathfrak p_\C)$ of $\mathfrak p_\C$ and $\mathrm{sym}\colon S(\mathfrak p_\C)\to U(\mathfrak p_\C)$ is the usual symmetrization map.
\end{rem}

\end{document}